\documentclass[11pt]{amsart}
\usepackage{amsfonts}
\usepackage{}
\usepackage{pifont}
\usepackage{mathtools}
\usepackage{esint}
\usepackage{cite}

\usepackage{amsmath}
\usepackage{amssymb}
\usepackage{latexsym}
\usepackage{amscd}
\usepackage{mathrsfs}
\usepackage[all]{xy}

\newdimen\AAdi%
\newbox\AAbo%
\def\AAk#1#2{\s_etbox\AAbo=\hbox{#2}\AAdi=\wd\AAbo\kern#1\AAdi{}}%
\def\AAr#1#2#3{\s_etbox\AAbo=\hbox{#2}\AAdi=\ht\AAbo\raise#1\AAdi\hbox{#3}}%
\font\tenmsb=msbm10 at 12pt \font\sevenmsb=msbm7 at 8pt
\font\fivemsb=msbm5 at 6pt
\newfam\msbfam
\textfont\msbfam=\tenmsb \scriptfont\msbfam=\sevenmsb
\scriptscriptfont\msbfam=\fivemsb

\newtheorem{theorem}{Theorem}

\newtheorem{remark}[theorem]{Remark}

\newtheorem{lemma}[theorem]{Lemma}

\numberwithin{equation}{section} \numberwithin{theorem}{section}

\renewcommand{\topmargin}{0cm}
\renewcommand{\oddsidemargin}{5mm}
\renewcommand{\evensidemargin}{5mm}
\renewcommand{\textwidth}{150mm}
\renewcommand{\textheight}{230mm}

\def\R{\mathbb R}

\def\Z{\mathbb Z}

\def\Z{\mathbb Z}

\def\na{\nabla}
\def\bn{\overline\nabla}

\def\f#1#2{\frac{#1}{#2}}

\def\a{\alpha}
\def\be{\beta}

\def\r{\Re_{I\!V}}

\def\p#1{\partial #1}

\def\de{\delta}
\def\De{\Delta}
\def\e{\eta}
\def\ep{\epsilon}

\def\G{\Gamma}

\def\la{\lambda}
\def\La{\Lambda}

\def\Om{\Omega}
\def\th{\theta}

\def\si{\sigma}
\def\Si{\Sigma}

\def\r{\rho}

\begin{document}

\title
{Liouville type theorems for minimal graphs over manifolds}

\author{Qi Ding}
\address{Shanghai Center for Mathematical Sciences, Fudan University, Shanghai 200438, China}
\email{dingqi@fudan.edu.cn}

\thanks{The author is supported by NSFC 11871156 and NSFC 11922106}

\begin{abstract}
Let $\Si$ be a complete Riemannian manifold with the volume doubling property and the uniform Neumann-Poincar$\mathrm{\acute{e}}$ inequality.
We show that any positive minimal graphic function on $\Si$ is a constant.
\end{abstract}

\maketitle

\section{Introduction}

Let $\Si$ denote a smooth complete non-compact Riemannian manifold with Levi-Civita connection $D$.
Let $\mathrm{div}_\Si$ be the divergence operator in terms of the Riemannian metric of $\Si$.
In this paper, we study the minimal hypersurface equation on $\Si$
\begin{equation}\label{u0}
\mathrm{div}_\Si\left(\f{Du}{\sqrt{1+|Du|^2}}\right)=0,
\end{equation}
which is a non-linear partial differential equation describing the minimal graph
$$M=\{(x,u(x))\in \Si\times\R|\, x\in\Si\}$$
over $\Si$. The equation \eqref{u0} is equivalent to that $u$ is harmonic on $M$, i.e.,
\begin{equation}\label{DeMu0}
\De_Mu=0,
\end{equation}
where $\De_M$ is the Laplacian on $M$.
The solution $u$ to \eqref{u0} is the height function of the minimal graph $M$ in $\Si\times\R$. Therefore we call $u$ a \emph{minimal graphic function} on $\Si$.

When $\Si$ is a Euclidean space $\R^n$, \eqref{u0} is exactly the famous minimal surface equation on $\R^n$.
In 1961, J. Moser \cite{M} derived Harnack's inequalities for uniformly elliptic equations, which imply Bernstein theorem for minimal graphs of bounded slope in all dimensions.
In 1969, Bombieri-De Giorgi-Miranda \cite{BGM} (see also \cite{GT}) showed interior gradient estimates for solutions to the minimal surface equation on $\R^n$, where the 2-dimensional case had already been obtained by Finn \cite{F}.
Using the gradient estimates, they get a Liouville type theorem in \cite{BGM} as follows.
\begin{theorem}\label{PMGFC0}
Any positive minimal graphic function on $\R^n$ is a constant.
\end{theorem}
Without the 'positive' condition in Theorem \ref{PMGFC0}, it is exact Bernstein theorem (see \cite{Fl,DG,Al,Si} and the counter-example in \cite{BDG}). Specially, any  minimal graphic function on $\R^n$ is affine for $n\le7$.

As the linear analogue of \eqref{u0} or \eqref{DeMu0}, harmonic functions have been studied successfully on manifolds of nonnegative Ricci curvature.
Yau \cite{Y} showed a Liouville theorem for harmonic functions:

\emph{Every positive harmonic function on a complete manifold of nonnegative Ricci curvature is a constant.}

Compared with this, it is natural to study Liouville type theorems for the solutions to \eqref{u0} on manifolds of nonnegative Ricci curvature.
Since minimal graphs in $\Si\times\R$ are area-minimizing, any positive minimal graphic function on a Riemann surface $\Si$ of nonnegative curvature is a constant from Fischer-Colbrie and Schoen \cite{FS}.
For the general dimension $n$, Rosenberg-Schulze-Spruck \cite{RSS} generalized Theorem \ref{PMGFC0}.
Specifically, they showed that any positive minimal graphical function on an $n$-dimensional complete manifold $\Si$ is a constant provided $\Si$ has nonnegative Ricci curvature and sectional curvature uniformly bounded from below.
Besides the minimal graphs of dimension $>7$ constructed by Bombieri-De Giorgi-Giusti \cite{BDG}, for all $n\ge4$ there are non totally geodesic minimal graphs over $n$-dimensional complete manifolds of positive sectional curvature \cite{DJX1}.
In the present paper, we obtain the following Liouville type theorem.
\begin{theorem}\label{PMGFC}
Any positive minimal graphic function on a complete manifold of nonnegative Ricci curvature is a constant.
\end{theorem}
In fact, Theorem \ref{PMGFC} is a consequence of a much more general result. In order to state this, let us recall the definition of some basic analytic inequalities on complete Riemannian manifold $\Si$.
Let $B_r(p)$ denote the geodesic ball in $\Si$ of the radius $r$ and centered at $p\in\Si$.
We call that $\Si$ has \emph{the volume doubling property}, if there exists a positive constant $C_D>1$ such that for all $p\in\Si$ and $r>0$
\begin{equation}\aligned\label{VD}
\mathcal{H}^n(B_{2r}(p))\le C_D\mathcal{H}^n(B_{r}(p)),
\endaligned
\end{equation}
where $\mathcal{H}^n(\cdot)$ denotes the $n$-dimensional Hausdorff measure.
We call that $\Si$ satisfies \emph{a uniform Neumann-Poincar$\mathrm{\acute{e}}$ inequality}, if there exists a positive constant $C_N\ge1$ such that for all $p\in\Si$, $r>0$ and $f\in W^{1,1}(B_r(p))$
\begin{equation}\aligned\label{NP}
\int_{B_r(p)}|f-\bar{f}_{p,r}|\le C_N r\int_{B_r(p)}|Df|,
\endaligned
\end{equation}
where $\bar{f}_{p,r}=\f1{\mathcal{H}^n(B_r(p))}\int_{B_r(p)}f$.

If $\Si$ has nonnegative Ricci curvature, then $\Si$ automatically satisfies the volume doubling property with doubling constant $C_D=2^n$ by Bishop-Gromov volume comparison theorem. From P. Buser \cite{Bu} or Cheeger-Colding \cite{CCo3}, $\Si$ satisfies a uniform Neumann-Poincar$\mathrm{\acute{e}}$ inequality with $C_N=C_N(n)<\infty$.
Now we can state our more general result compared with Theorem \ref{PMGFC} as follows.
\begin{theorem}\label{main**}
Let $\Si$ be an $n$-dimensional complete Riemannian manifold with \eqref{VD}\eqref{NP}.
If $u$ is a positive minimal graphic function on $\Si$,
then $u$ is a constant.
\end{theorem}
The conditions \eqref{VD}\eqref{NP} we introduced in the above theorem are inspired by Colding-Minicozzi \cite{CM2}, where they studied harmonic functions of polynomial growth on complete Riemannian manifolds with \eqref{VD}\eqref{NP}, which resolved Yau's conjecture.
The key ingredient in the proof of Theorem \ref{main**} is to get the Sobolev inequality and the Neumann-Poincar$\mathrm{\acute{e}}$ inequality for the positive monotonic $C^1$-functions of the minimal graphic function $u$. And this is sufficient to carry out De Giorgi-Nash-Moser iteration for the positive monotonic $C^1$-functions of $u$.

\section{Analytic inequalities on $\Si$ and $\Si\times\R$}

Let $\Si$ be an $n$-dimensional complete manifold with Riemannian metric $\si$ and the Levi-Civita connection $D$.
Suppose that $\Si$ satisfies the volume doubling property \eqref{VD} and the uniform Neumann-Poincar$\mathrm{\acute{e}}$ inequality \eqref{NP}.
From \eqref{VD}, one has
\begin{equation}\aligned\nonumber
\mathcal{H}^n(B_{R}(p))\le C_D\mathcal{H}^n\left(B_{\f R2}(p)\right)\le\cdots\le C_D^k\mathcal{H}^n\left(B_{\f R{2^k}}(p)\right)=2^{k\log_2{C_D}}\mathcal{H}^n\left(B_{\f R{2^k}}(p)\right)
\endaligned
\end{equation}
for any $R>0$ and $k\in\Z^+$.
Hence, there is a constant $\a=\log_2{C_D}>0$ such that
\begin{equation}\aligned\label{VG}
\mathcal{H}^n\left(B_{r}(p)\right)\ge \f1{C_D}\mathcal{H}^n\left(B_{R}(p)\right)\left(\f rR\right)^{\a}
\endaligned
\end{equation}
for all $r\in(0,R)$.
Without loss of generality, we assume that $C_D\ge4$, then
\begin{equation}\aligned\label{alog2CD}
\a=\log_2{C_D}\ge2.
\endaligned
\end{equation}

From \eqref{VD} and 5-lemma, there is a constant $\La_D\ge1$ depending on $C_D$ such that
\begin{equation}\aligned\nonumber
\mathcal{H}^n(B_{R}(p)\setminus B_{R-r}(p))\le \La_D \mathcal{H}^n(B_{R-r}(p)\setminus B_{R-2r}(p))
\endaligned
\end{equation}
for all $0<r<R/2$, which implies
\begin{equation}\aligned\label{BRR-rR-2r}
\mathcal{H}^n(B_{R}(p)\setminus B_{R-r}(p))\le \f{\La_D}{\La_D+1} \mathcal{H}^n(B_{R}(p)\setminus B_{R-2r}(p)).
\endaligned
\end{equation}
For all $0<r<R$, there is an integer $k\ge0$ such that $2^{-k-1}<r/R\le2^{-k}$. Then from \eqref{BRR-rR-2r}, we have
\begin{equation}\aligned\label{HnBRR-rp}
&\mathcal{H}^n(B_{R}(p)\setminus B_{R-r}(p))\le \left(\f{\La_D}{\La_D+1}\right)^k \mathcal{H}^n(B_{R}(p)\setminus B_{R-2^kr}(p))\\
\le&2^{-k\log_2(1+1/\La_D)}\mathcal{H}^n(B_{R}(p))\le\left(\f {2r}R\right)^{\log_2(1+1/\La_D)}\mathcal{H}^n(B_{R}(p)).
\endaligned
\end{equation}

From \eqref{NP}, we have
\begin{equation}\aligned\label{NPOm}
\min\left\{\mathcal{H}^n(\Om),\mathcal{H}^n(B_{r}(p)\setminus\Om)\right\}\le C_Nr\mathcal{H}^{n-1}(B_r(p)\cap\p\Om)
\endaligned
\end{equation}
immediately for any open set $\Om$ in $B_r(p)$ with rectifiable boundary.
Combining \eqref{VD} and \eqref{NP}, one can get an isoperimetric inequality on $\Si$ (see the appendix for a self-contained proof).
Namely, there exists a constant $C_S\ge1$ depending only on $C_D,C_N$ such that
for any $p\in\Si$, $r>0$, we have
\begin{equation}\aligned
\left(\mathcal{H}^n(\Om)\right)^{1-\f1\a}\le C_S\left(\mathcal{H}^n(B_{r}(p))\right)^{-\f1\a}r\mathcal{H}^{n-1}(\p\Om)
\endaligned
\end{equation}
for any open set $\Om\subset B_r(p)$ with rectifiable boundary. By a standard argument (see \cite{SY} for instance),
for $f\in W^{1,1}_{0}(B_r(p))$ we have
\begin{equation}\aligned\label{SobSi}
\left(\int_{B_r(p)}|f|^{\f{\a}{\a-1}}\right)^{\f{\a-1}\a}\le C_S\left(\mathcal{H}^n(B_r(p))\right)^{-\f1\a}r\int_{B_{r}(p)}|Df|.
\endaligned
\end{equation}
From \eqref{VD} and \eqref{NP} on $\Si$, we can get the Sobolev-Poincar$\mathrm{\acute{e}}$ inequality on $\Si$ (see Theorem 1 in \cite{HK} for instance).
Namely, up to select the constants $\a\ge2$ and $C_S\ge1$, for any $p\in\Si$, $r>0$, and $f\in W^{1,1}(B_r(p))$, we have
\begin{equation}\aligned\label{SobP}
\left(\int_{B_r(p)}|f-\bar{f}_{p,r}|^{\f{\a}{\a-1}}\right)^{\f{\a-1}\a}\le C_S\left(\mathcal{H}^n(B_r(p))\right)^{-\f1\a}r\int_{B_{r}(p)}|Df|,
\endaligned
\end{equation}
where $\bar{f}_{p,r}=\f1{\mathcal{H}^n(B_r(p))}\int_{B_r(p)}f$.
The Sobolev-Poincar$\mathrm{\acute{e}}$ inequality \eqref{SobP} implies the following isoperimetric type inequality, compared with \eqref{NPOm}.
\begin{lemma}\label{ISOPOIN}
For any open set $\Om$ in $B_r(p)$ with rectifiable boundary, we have
\begin{equation}\aligned
\mathcal{H}^{n-1}(B_r(p)\cap\p\Om)\ge
\f{\left(\mathcal{H}^n(B_r(p))\right)^{\f{1}\a}}{2^{\f1\a}C_S r}\left(\min\left\{\mathcal{H}^n(\Om),\mathcal{H}^n(B_{r}(p)\setminus\Om)\right\}\right)^{\f{\a-1}\a}.
\endaligned
\end{equation}
\end{lemma}
\begin{proof}
Denote $\Om_*=B_{r}(p)\setminus\Om$. Let $f$ be a function of bounded variation defined by $f\equiv 0$ on $\Om$, and $f\equiv \mathcal{H}^n(B_r(p))$ on $\Om_*$.
For any $\ep>0$, let $f_\ep$ be a Lipschitz function defined on $B_r(p)$ by letting $f_\ep\equiv 0$ on $\Om$, $f_\ep\equiv t\mathcal{H}^n(B_r(p))/\ep$ on $\{x\in B_r(p)|\ d(x,\Om)=t\}$ for any $t\in(0,\ep]$ and $f_\ep\equiv \mathcal{H}^n(B_r(p))$ on $\{x\in B_r(p)|\ d(x,\Om)>\ep\}$, where $d(\cdot,\Om)$ is the distance function to $\Om$ on $\Si$.
Then $f_\ep$ satisfies
$$\lim_{\ep\rightarrow0}\int_{B_r(p)}f_\ep=\int_{B_r(p)} f=\mathcal{H}^n(B_r(p))\mathcal{H}^n(\Om_*).$$
Using \eqref{SobP} for $f_\ep$, and then letting $\ep\rightarrow0$ implies
\begin{equation}\aligned
&\left(\left(\mathcal{H}^n(\Om)\right)^{\f{\a}{\a-1}}\mathcal{H}^n(\Om_*)+\left(\mathcal{H}^n(\Om_*)\right)^{\f{\a}{\a-1}}\mathcal{H}^n(\Om)\right)^{\f{\a-1}\a}\\
\le& C_S\left(\mathcal{H}^n(B_r(p))\right)^{-\f1\a}r\mathcal{H}^{n-1}(B_{r}(p)\cap\p\Om)\mathcal{H}^n(B_r(p)).
\endaligned
\end{equation}
Without loss of generality, we assume $\mathcal{H}^n(\Om_*)\le \mathcal{H}^n(\Om)$. Then
\begin{equation}\aligned
&C_S\left(\mathcal{H}^n(B_r(p))\right)^{1-\f1\a}r\mathcal{H}^{n-1}(B_{r}(p)\cap\p\Om)\\
\ge&\left(\mathcal{H}^n(\Om_*)\right)^{\f{\a-1}\a}\mathcal{H}^n(\Om)\left(1+\left(\f{\mathcal{H}^n(\Om_*)}{\mathcal{H}^n(\Om)}\right)^{\f{1}{\a-1}}\right)^{\f{\a-1}\a}\\
\ge&\left(\mathcal{H}^n(\Om_*)\right)^{\f{\a-1}\a}\mathcal{H}^n(\Om)2^{-\f1\a}\left(1+\f{\mathcal{H}^n(\Om_*)}{\mathcal{H}^n(\Om)}\right)\\
=&2^{-\f1\a}\left(\mathcal{H}^n(\Om_*)\right)^{\f{\a-1}\a}\mathcal{H}^n(B_r(p)),
\endaligned
\end{equation}
where in the third inequality we have used $2^{\f1\a}(1+t)^{\f{\a-1}\a}\ge1+t\ge1+t^{\a-1}$ for any $t\in[0,1]$ and $\a\ge2$.
This completes the proof.
\end{proof}

Let $\Si\times\R$ be the product manifold with the flat product metric $\si+ds^2$.
For any $\bar{p}=(p,t_p)\in\Si\times\R$, let $B_r(\bar{p})$ denote the geodesic ball in $\Si\times\R$ of the radius $r$ and centered at $\bar{p}$.
Then the volume of the geodesic ball $B_r(\bar{p})$ satisfies
\begin{equation}\aligned
\mathcal{H}^{n+1}(B_r(\bar{p}))=2\int_{x\in B_r(p)}\sqrt{r^2-d(x,p)^2},
\endaligned
\end{equation}
where $d(\cdot,p)$ is the distance function on $\Si$ from $p$.
Combining \eqref{VG}, we have
\begin{equation}\aligned\label{VOLBrbp}
2r\mathcal{H}^n(B_r(p))\ge\mathcal{H}^{n+1}(B_r(\bar{p}))\ge2\int_{B_{\f r2}(p)}\f{\sqrt{3}}2r=\sqrt{3}r\mathcal{H}^n\left(B_{\f r2}(p)\right)\ge \f{\sqrt{3}r}{2^{\a}C_D} \mathcal{H}^n\left(B_r(p)\right).
\endaligned
\end{equation}
Then combining \eqref{VG} we have
\begin{equation}\aligned\label{VG*}
&\mathcal{H}^{n+1}\left(B_{R}(\bar{p})\right)\le2R\mathcal{H}^n\left(B_{R}(p)\right)\le 2RC_D\left(\f Rr\right)^{\a}\mathcal{H}^n\left(B_{r}(p)\right)\\
\le&2RC_D\left(\f Rr\right)^{\a}\f{2^{\a}C_D}{\sqrt{3}r}\mathcal{H}^{n+1}\left(B_{r}(\bar{p})\right)=\f{2^{\a+1}}{\sqrt{3}}C_D^2\left(\f Rr\right)^{\a+1}\mathcal{H}^{n+1}\left(B_{r}(\bar{p})\right)
\endaligned
\end{equation}
for all $r\in(0,R)$.

Let $\pi$ denote the projection from $\Si\times\R$ onto $\Si$.
For any $\bar{x}=(x,t_x)\in\Si\times\R$, let $C_{\bar{x},r,t}$ be a cylinder in $\Si\times\R$ defined by
\begin{equation}\aligned
C_{\bar{x},r,t}=B_r(x)\times(t_x-t,t_x+t).
\endaligned
\end{equation}
Let $C_{\bar{x},r}=C_{\bar{x},r,r}$ for convenience.
\begin{lemma}\label{CheegerC}
For any open set $V$ in $C_{\bar{p},r}$ with rectifiable boundary, there is an absolute constant $c>0$ such that
\begin{equation}\aligned\label{HnpVdeDN}
C_Nr\mathcal{H}^{n}(C_{\bar{p},r}\cap\p V)\ge
\f1c\min\left\{\mathcal{H}^{n+1}(V),\mathcal{H}^{n+1}(C_{\bar{p},r}\setminus V)\right\}.
\endaligned
\end{equation}
\end{lemma}
\begin{proof}
Let $\G=\p V\cap C_{\bar{p},r}$.
Let $\de$ be a positive constant with
\begin{equation}\aligned\label{piGde***}
\mathcal{H}^n(\pi(\G))=\de\mathcal{H}^n\left(B_r(p)\right).
\endaligned
\end{equation}
Suppose that $\de\le\f18$, or else
$$\mathcal{H}^n(\G)\ge\mathcal{H}^n(\pi(\G))\ge\f18\mathcal{H}^n\left(B_r(p)\right),$$
which completes the proof. If
\begin{equation}\aligned\label{piVpiG}
\min\{\mathcal{H}^{n}(\pi(V)\setminus\pi(\G)),\mathcal{H}^{n}(\pi(C_{\bar{p},r}\setminus V)\setminus\pi(\G))\}\le\f18\mathcal{H}^n\left(B_r(p)\right),
\endaligned
\end{equation}
then using \eqref{piGde***} we have
$$\min\{\mathcal{H}^{n}(\pi(V)),\mathcal{H}^{n}(\pi(C_{\bar{p},r}\setminus V))\}\le\f14\mathcal{H}^n\left(B_r(p)\right).$$
Now we assume
\begin{equation}\aligned\label{AsspiV}
\mathcal{H}^{n}(\pi(V))\le\f14\mathcal{H}^n\left(B_r(p)\right).
\endaligned
\end{equation}
Let $V_t=V\cap(\Si\times\{t\})$ for all $|t-t_p|<r$. Clearly, $\mathcal{H}^{n}(V_t)\le\mathcal{H}^{n}(\pi(V))\le\f14\mathcal{H}^n\left(B_r(p)\right)$.
From the uniform Neumann-Poincar$\mathrm{\acute{e}}$ inequality \eqref{NP}, we get
\begin{equation}\aligned\label{Vttpr}
\mathcal{H}^{n}(V_t)\le C_Nr\mathcal{H}^{n-1}(\p V\cap(B_r(p)\times\{t\})).
\endaligned
\end{equation}
Combining the co-area formula we have
\begin{equation}\aligned\label{intVttpr}
&\mathcal{H}^{n+1}(V)=\int_{t_p-r}^{t_p+r}\mathcal{H}^{n}(V_t)dt\\
\le& C_Nr\int_{t_p-r}^{t_p+r}\mathcal{H}^{n-1}(\p V\cap(B_r(p)\times\{t\}))dt\le C_Nr\mathcal{H}^{n}(\p V\cap C_{\bar{p},r}).
\endaligned
\end{equation}
Clearly, the above inequality is true provided the assumption \eqref{AsspiV} is replaced by $\mathcal{H}^{n}(\pi(C_{\bar{p},r}\setminus V))\le\f14\mathcal{H}^n\left(B_r(p)\right)$.
Hence, we always have
\begin{equation}\aligned\label{intVCV}
\min\{\mathcal{H}^{n+1}(V),\mathcal{H}^{n+1}(C_{\bar{p},r}\setminus V)\}\le C_Nr\mathcal{H}^{n}(\p V\cap C_{\bar{p},r}).
\endaligned
\end{equation}

If \eqref{piVpiG} fails, we have
\begin{equation}\aligned
\min\{\mathcal{H}^{n}(\pi(V)\setminus\pi(\G)),\mathcal{H}^{n}(\pi(C_{\bar{p},r}\setminus V)\setminus\pi(\G))\}\ge\f18\mathcal{H}^n\left(B_r(p)\right),
\endaligned
\end{equation}
which implies
\begin{equation}\aligned
\min\{\mathcal{H}^{n}(V_t),\mathcal{H}^{n}(B_r(p)\times\{t\}\setminus V_t)\}\ge\f18\mathcal{H}^n\left(B_r(p)\right)
\endaligned
\end{equation}
for all $|t-t_p|<r$. With \eqref{NPOm}, we have
\begin{equation}\aligned
\mathcal{H}^n\left(\p V\cap(B_r(p)\times\{t\})\right)\ge\f1{C_Nr}\min\{\mathcal{H}^{n}(V_t),\mathcal{H}^{n}(B_r(p)\times\{t\}\setminus V_t)\}\ge\f{\mathcal{H}^n\left(B_r(p)\right)}{8C_Nr}.
\endaligned
\end{equation}
Combining the co-area formula, we get
\begin{equation}\aligned
C_Nr\mathcal{H}^{n}(\p V\cap C_{\bar{p},r})\ge C_Nr\int_{t_p-r}^{t_p+r}\mathcal{H}^n\left(\p V\cap(B_r(p)\times\{t\})\right)dt\ge\f{r}{4}\mathcal{H}^n\left(B_r(p)\right),
\endaligned
\end{equation}
which implies \eqref{HnpVdeDN}. We complete the proof.
\end{proof}

From Lemma \ref{CheegerC}, we have a Neumann-Poincar$\mathrm{\acute{e}}$ inequality in $\Si\times\R$ as follows.
\begin{lemma}\label{CheegerC*}
\begin{equation}\aligned
\int_{B_{r}(\bar{p})}|f-\bar{f}_{\bar{p},r}|\le 2cC_N r\int_{B_{\sqrt{2}r}(\bar{p})}|\na f|
\endaligned
\end{equation}
for all function $f\in W^{1,1}_{loc}(\Si\times\R)$,
where $\bar{f}_{\bar{p},r}=\f1{\mathcal{H}^{n+1}(B_r(\bar{p}))}\int_{B_r(\bar{p})}f$.
\end{lemma}
\begin{proof}
Let
$$V^+_{s,t}=\{x\in B_{s}(\bar{p})|\, f(x)>\bar{f}_{\bar{p},r}+t\},$$
and
$$V^-_{s,t}=\{x\in B_{s}(\bar{p})|\, f(x)<\bar{f}_{\bar{p},r}+t\}$$
for all $s\in(0,2r]$ and $t\in\R$.
Without loss of generality, we assume $\mathcal{H}^{n+1}(V^+_{r,0})\le\mathcal{H}^{n+1}(V^-_{r,0})$. Then
\begin{equation}\aligned
\mathcal{H}^{n+1}(V^+_{r,t})\le\f12\mathcal{H}^{n+1}(B_{r}(\bar{p}))
\endaligned
\end{equation}
for any $t\ge0$.
From Lemma \ref{CheegerC}, there holds
\begin{equation}\aligned
\mathcal{H}^{n}\left(\p V^+_{\sqrt{2}r,t}\cap B_{\sqrt{2}r}(\bar{p})\right)\ge\f1{cC_Nr}\mathcal{H}^{n+1}(V^+_{r,t}).
\endaligned
\end{equation}
From co-area formula, we have
\begin{equation}\aligned
\int_{V^+_{r,0}}\left(f-\bar{f}_{\bar{p},r}\right)=&\int_0^\infty\mathcal{H}^{n+1}(V^+_{r,t})dt\le cC_N r\int_0^\infty\mathcal{H}^{n}\left(\p V^+_{\sqrt{2}r,t}\cap B_{\sqrt{2}r}(\bar{p})\right)dt\\
\le& cC_N r\int_{B_{\sqrt{2}r}(\bar{p})}|\na f|.
\endaligned
\end{equation}
Hence
\begin{equation}\aligned
&\int_{B_{r}(\bar{p})}\left|f-\bar{f}_{\bar{p},r}\right|=\int_{V^+_{r,0}}\left(f-\bar{f}_{\bar{p},r}\right)-\int_{V^-_{r,0}}\left(f-\bar{f}_{\bar{p},r}\right)\\
=&2\int_{V^+_{r,0}}\left(f-\bar{f}_{\bar{p},r}\right)\le 2cC_N r\int_{B_{\sqrt{2}r}(\bar{p})}|\na f|.
\endaligned
\end{equation}
This completes the proof.
\end{proof}

Combining \eqref{VG*}, Lemma \ref{CheegerC*} and Theorem 1 in \cite{HK}, up to select the constants $\a\ge2$ and $C_S\ge1$, for any $\bar{p}\in\Si\times\R$, $r>0$, and $f\in W^{1,1}(B_r(\bar{p}))$, we have
\begin{equation}\aligned
\left(\int_{B_r(\bar{p})}|f-\bar{f}_{\bar{p},r}|^{\f{\a+1}{\a}}\right)^{\f{\a}{\a+1}}\le C_S\left(\mathcal{H}^{n+1}(B_r(\bar{p}))\right)^{-\f1{\a+1}}r\int_{B_{r}(\bar{p})}|Df|,
\endaligned
\end{equation}
where $\bar{f}_{\bar{p},r}=\f1{\mathcal{H}^{n+1}(B_r(\bar{p}))}\int_{B_r(\bar{p})}f$.
From the argument of Lemma \ref{ISOPOIN}, we have
\begin{equation}\aligned\label{SP11}
\mathcal{H}^{n}(B_r(\bar{p})\cap\p V)\ge
\f{\left(\mathcal{H}^{n+1}(B_r(\bar{p}))\right)^{\f{1}{\a+1}}}{2^{\f1{\a+1}}C_S r}\left(\min\left\{\mathcal{H}^{n+1}(V),\mathcal{H}^{n+1}(B_{r}(\bar{p})\setminus V)\right\}\right)^{\f{\a}{\a+1}}
\endaligned
\end{equation}
for any open set $V$ in $B_r(\bar{p})$ with rectifiable boundary. In particular,
\begin{equation}\aligned\label{SP10}
&\mathcal{H}^{n}(B_r(\bar{p})\cap\p V)\\
\ge&
\f{\left(\mathcal{H}^{n+1}(B_r(\bar{p}))\right)^{\f{1}{\a+1}}}{2^{\f1{\a+1}}C_S r}\min\left\{\mathcal{H}^{n+1}(V),\mathcal{H}^{n+1}(B_{r}(\bar{p})\setminus V)\right\}\left(\f{\mathcal{H}^{n+1}(B_r(\bar{p}))}{2}\right)^{-\f{1}{\a+1}}\\
=&\f{1}{C_S r}\min\left\{\mathcal{H}^{n+1}(V),\mathcal{H}^{n+1}(B_{r}(\bar{p})\setminus V)\right\}.
\endaligned
\end{equation}
From the argument in the appendix, there holds the isoperimetric inequality and the Sobolev inequality on $\Si\times\R$ from the doubling property \eqref{VOLBrbp} and \eqref{SP10}.
Namely, up to choose the constant $C_S\ge1$ depending only on $C_D,C_N$, we have
\begin{equation}\aligned\label{Soba}
\left(\mathcal{H}^{n+1}(V)\right)^{\f{\a}{\a+1}}\le C_S\left(\mathcal{H}^{n+1}(B_r(\bar{p}))\right)^{-\f1{\a+1}}r\mathcal{H}^{n+1}(\p V)
\endaligned\end{equation}
for any open set $V$ in $B_r(\bar{p})$ with rectifiable boundary.,

\section{Sobolev and Neumann-Poincar$\mathrm{\acute{E}}$ inequalities on minimal graphs}

Let $M$ be a minimal graph over the geodesic ball $B_R(p)\subset\Si$ with the graphic function $u$, i.e., $u$ satisfies \eqref{u0} on $B_R(p)$. Denote $\bar{p}=(p,u(p))\in M$.
Similar to the Euclidean case, $M$ is an area-minimizing hypersurface in $B_R(p)\times\R$ (see \cite{R} or Lemma 2.1 in \cite{DJX2}).
Namely, for any rectifiable hypersurface $S$ in $B_R(p)\times\R$ with boundary $\p S=\p M$, we have
\begin{equation}\aligned
\mathcal{H}^n(M)\le \mathcal{H}^n(S).
\endaligned
\end{equation}
For the fixed $R>0$, let
\begin{equation}\aligned\label{Om+-}
\Om_\pm=\{(x,t)\in B_R(p)\times\R|\, t>(<)u(x)\}.
\endaligned
\end{equation}
\begin{lemma}\label{lbdMG*}
There is a constant $\be\in(0,\f12]$ depending only on $C_D,C_N$ such that for any $(n+1)$-dimensional ball $B_r(z)\subset B_R(p)\times\R$ with $z\in M$ we have
$$\mathcal{H}^{n+1}(\Om_+\cap B_r(z))\ge \be\mathcal{H}^{n+1}(B_r(z)).$$
\end{lemma}
\begin{proof}
Let $\Om_{s,z}=\Om_+\cap\p B_s(z)$ for any $0<s<r$. Then
$$\p \Om_{s,z}=\p \Om_+\cap\p B_s(z)=\p(B_s(z)\cap\p \Om_+).$$
Since $\p \Om_+\cap B_s(z)$ is area-minimizing in $B_R(p)\times\R$, then
\begin{equation}\aligned
\mathcal{H}^{n}(\Om_{s,z})\ge\mathcal{H}^{n}(B_s(z)\cap\p \Om_+).
\endaligned
\end{equation}
By the isoperimetric inequality \eqref{Soba}, for any $0<t<r$ we have
\begin{equation}\aligned
&\left(\mathcal{H}^{n+1}(\Om_+\cap B_t(z))\right)^{\f{\a}{\a+1}}\le C_S\left(\mathcal{H}^{n+1}(B_t(z))\right)^{-\f1{\a+1}}t\mathcal{H}^{n}(\p(\Om_+\cap B_t(z))).
\endaligned\end{equation}
Since
$$\mathcal{H}^{n}(\p(\Om_+\cap B_t(z)))=\mathcal{H}^{n}(\Om_{t,z})+\mathcal{H}^{n}(B_t(z)\cap \p \Om_+)\le2\mathcal{H}^{n}(\Om_{t,z}),$$
then
\begin{equation}
\left(\mathcal{H}^{n+1}(B_t(z))\right)^{\f1{\a+1}}\left(\mathcal{H}^{n+1}(\Om_+\cap B_t(z))\right)^{\f{\a}{\a+1}}\le 2C_St\mathcal{H}^{n}(\Om_{t,z})
\end{equation}
for any $t\in(0,r)$, which means
\begin{equation}\label{ODEEsx}
\left(\mathcal{H}^{n+1}(B_t(z))\right)^{\f1{\a+1}}\left(\int_0^t\mathcal{H}^{n}(\Om_{s,z})ds\right)^{\f{\a}{\a+1}}\le 2C_St\mathcal{H}^{n}(\Om_{t,z}).
\end{equation}
As $z\in M$, one has $\int_0^t\mathcal{H}^{n}(\Om_{s,z})ds>0$ for each $t>0$. From \eqref{ODEEsx}, we get
\begin{equation}
\f{\p}{\p t}\left(\int_0^t\mathcal{H}^{n}(\Om_{s,z})ds\right)^{\f{1}{\a+1}}\ge\f{1}{2(\a+1) C_St}\left(\mathcal{H}^{n+1}(B_t(z))\right)^{\f1{\a+1}}
\end{equation}
on $(0,r)$,
which gives
\begin{equation}\nonumber
\aligned
\left(\int_0^r\mathcal{H}^{n}(\Om_{s,z})ds\right)^{\f{1}{\a+1}}\ge&\f{1}{2(\a+1) C_S}\left(\mathcal{H}^{n+1}(B_{\f r2}(z))\right)^{\f1{\a+1}}\int_{\f r2}^r\f1t dt\\
=&\f{\log 2}{2(\a+1) C_S}\left(\mathcal{H}^{n+1}(B_{\f r2}(z))\right)^{\f1{\a+1}}.
\endaligned
\end{equation}
In particular,
\begin{equation}
\mathcal{H}^{n+1}(\Om_+\cap B_r(z))=\int_0^r\mathcal{H}^{n}(\Om_{s,z})ds\ge\left(\f{\log2}{2(\a+1) C_S}\right)^{\a+1}\mathcal{H}^{n+1}(B_{\f r2}(z)),
\end{equation}
and this completes the proof.
\end{proof}

From the above lemma, we also have
$$\mathcal{H}^{n+1}(\Om_-\cap B_r(z))\ge \be\mathcal{H}^{n+1}(B_r(z))$$
for any $B_r(z)\subset B_R(p)\times\R$ with $z\in M$.
Combining \eqref{SP10}, we can give a lower bound of the volume of minimal graphs as follows.
\begin{equation}\aligned\label{lbdMG}
\mathcal{H}^{n}(M\cap B_r(z))\ge\f{\be}{C_S r}\mathcal{H}^{n+1}(B_r(z)).
\endaligned
\end{equation}
Recall $C_{\bar{x},r}=B_r(x)\times(t_x-r,t_x+r)$ for $\bar{x}=(x,t_x)\in M$, and $\pi$ denotes the projection from $\Si\times\R$ onto $\Si$.
From Lemma \ref{lbdMG*}, we have
\begin{equation}\aligned\label{ldbOm+Cbxr}
\mathcal{H}^{n+1}(\Om_+\cap C_{z,r})\ge\mathcal{H}^{n+1}(\Om_+\cap B_r(z))\ge \be\mathcal{H}^{n+1}(B_r(z)).
\endaligned
\end{equation}
Since
$$\Om_+\cap (B_r(\pi(z))\times\{t_1\})\subset \Om_+\cap (B_r(\pi(z))\times\{t_2\})$$
for all $t_1<t_2$, then from co-area formula and \eqref{ldbOm+Cbxr}, we have
\begin{equation}\aligned\label{lbdcylinderM}
&\mathcal{H}^{n}\left(\Om_+\cap (B_r(\pi(z))\times\{t\})\right)\ge\mathcal{H}^{n}\left(\Om_+\cap (B_r(\pi(z))\times\{r\})\right)\\
\ge&\f1{2r}\mathcal{H}^{n+1}(\Om_+\cap C_{z,r})\ge \f{\be}{2r}\mathcal{H}^{n+1}(B_r(z))
\endaligned
\end{equation}
for any $t\ge r$.

Now let us define several open sets in $\Si\times\R$ as follows.
For any $r>0$, $t\in\R$ and $\bar{x}=(x,t_x)\in\Si\times\R$, we define
$$\mathfrak{D}_{\bar{x},r}=\{(y,s)\in\Si\times\R|\, d(x,y)+|s-t_x|<r\},$$
$$\mathfrak{D}_{\bar{x},r}^t=\{(y,s)\in\Si\times\R|\, d(x,y)+|s-t_x|<r,\, s>t+t_x\},$$
and
$$\mathfrak{D}_{\bar{x},r}^\pm=\{(y,s)\in\Si\times\R|\, d(x,y)+|s-t_x|<r,\, s>(<)t_x\}.$$
Then
$$B_{r}(\bar{x})\subset\mathfrak{D}_{\bar{x},\sqrt{2}r}\subset B_{\sqrt{2}r}(\bar{x}).$$
For any $n$-rectifiable set $\Om\subset\p\mathfrak{D}_{\bar{x},r}^+\setminus (B_r(x)\times\{t_x\})$, one has
\begin{equation}\aligned
\mathcal{H}^n(\Om)=\sqrt{2}\mathcal{H}^n(\pi(\Om)).
\endaligned
\end{equation}

From now on, we assume that the graphic function $u$ of the minimal graph $M$ is not a constant in this section.
Since $M$ is area-minimizing in $B_r(p)\times\R$, then
\begin{equation}\aligned
\mathcal{H}^n(M\cap\mathfrak{D}_{\bar{x},r})\le\f12\mathcal{H}^n(\p \mathfrak{D}_{\bar{x},r})=\sqrt{2}\mathcal{H}^n(B_r(x)).
\endaligned
\end{equation}
Combining \eqref{VOLBrbp} and \eqref{lbdMG}, there is a constant $\be_*\in(0,\be]$ depending only on $C_D,C_N$ such that
\begin{equation}\aligned\label{UplowM}
\be_*\mathcal{H}^{n}(B_r(x))\le\mathcal{H}^n\left(M\cap\mathfrak{D}_{\bar{x},r}\right)\le \sqrt{2}\mathcal{H}^{n}(B_r(x))
\endaligned
\end{equation}
for all $B_r(x)\subset B_R(p)$.
From \eqref{VG}\eqref{VOLBrbp} and Lemma \ref{lbdMG*}, there is a constant $\hat{\be}>0$ depending only on $C_D,C_N$ such that
\begin{equation}\aligned\label{Om+mfDxr}
\mathcal{H}^{n+1}(\Om^+\cap \mathfrak{D}_{\bar{x},r})\ge\hat{\be}r\mathcal{H}^n(B_r(x)).
\endaligned
\end{equation}

Let us prove an isoperimetric type inequality on $M$.
\begin{lemma}\label{Sob*}
There is a constant $\th$ depending only on $C_D,C_N$ such that for any $\bar{x}=(x,u(x))$ with $B_{2r}(x)\subset B_{R}(p)$ and $\tau\in(0,r)$, we have
\begin{equation}\aligned\label{Etlessth}
\left(\mathcal{H}^n(M\cap\mathfrak{D}_{\bar{x},r})\right)^{\f1\a}\left(\mathcal{H}^n(E_t\cap\mathfrak{D}_{\bar{x},r})\right)^{\f{\a-1}\a}\le \th\left(r\mathcal{H}^{n-1}(\p E_t\cap \mathfrak{D}_{\bar{x},r+\tau})+\f r{\tau}\mathcal{H}^n(E_t\cap\mathfrak{D}_{\bar{x},r}))\right),
\endaligned
\end{equation}
where $E_t=M\cap \mathfrak{D}_{\bar{x},2r}^t$ or $E_t=M\cap\left(\mathfrak{D}_{\bar{x},2r}\setminus \mathfrak{D}_{\bar{x},2r}^t\right)$ for any $t\in(-r,r)$.
\end{lemma}
\begin{proof}
Using $u-u(x)$ instead of $u$, we can assume $u(x)=0$.
By considering $-u,E_{-t}$ instead of $u,E_t$, respectively, we only need to show the case $E_t=M\cap \mathfrak{D}_{\bar{x},2r}^t$ for any $t\in(-r,r)$.

From \eqref{HnBRR-rp}, there exists a  constant $\be'\in(0,1)$ depending only on $C_D,C_N$ such that
\begin{equation}\aligned\label{HnBr1-hbebe}
\mathcal{H}^{n}\left(B_r(x)\setminus B_{(1-\be')r}(x)\right)\le\f{\hat{\be}}{16}\mathcal{H}^n(B_r(x)),
\endaligned
\end{equation}
where $\hat{\be}$ is the constant in \eqref{Om+mfDxr}.
From \eqref{VG} and \eqref{UplowM}, for $|t|>0$ we have
\begin{equation}\aligned\nonumber
\mathcal{H}^n\left(M\cap\mathfrak{D}_{\bar{x},|t|}\right)\ge\be_*\mathcal{H}^{n}(B_{|t|}(x))\ge\f{\be_*}{C_D}\mathcal{H}^{n}(B_r(x))\left(\f{|t|}r\right)^\a
\ge\f{\sqrt{2}\be_*}{2C_D}\mathcal{H}^{n}(M\cap\mathfrak{D}_{\bar{x},r})\f{|t|^\a}{r^\a}.
\endaligned
\end{equation}
Hence, we only need to prove \eqref{Etlessth} for $r>t\ge-\min\{\f12\tau,\f18\hat{\be}r,\be'r\}.$
Let $\de$ be a positive constant satisfying
\begin{equation}\aligned\label{HnEtDxr}
\mathcal{H}^n(E_t\cap\mathfrak{D}_{\bar{x},r}))=\de\left(\f {\tau}r\right)^\a\mathcal{H}^n(B_r(x)).
\endaligned
\end{equation}
We can assume $\de<\f18\hat{\be}$, or else we have complete the proof combining \eqref{UplowM}.

Let $U^-_t$ be a subset in $B_{r-t}(x)\times\{t\}$ defined by
$$U^-_t=\Om^-\cap\left(B_{r-t}(x)\times\{t\}\right),$$
and
\begin{equation}\aligned
U^+_t=\Om^+\cap\left(B_{r-t}(x)\times\{t\}\right)=B_{r-t}(x)\times\{t\}\setminus \overline{U^-_t}.
\endaligned
\end{equation}
We claim
\begin{equation}\aligned\label{U+low}
\mathcal{H}^n(U^+_t)\ge\de_*\mathcal{H}^n(B_{r-t}(x))
\endaligned
\end{equation}
for some positive constant $\de_*$ depending only on $C_D,C_N$.
The proof of the claim is divided into 2 cases as follows.
\begin{itemize}
  \item Case 1: $\mathcal{H}^n(\pi(M\cap \mathfrak{D}_{\bar{x},r}))\le\f14\hat{\be} \mathcal{H}^n(B_r(x))$.
  Let $\Om^+_{\bar{x},r}=\{(x,s)\in\Om^+\cap \mathfrak{D}_{\bar{x},r}|\, x\in \pi(M\cap \mathfrak{D}_{\bar{x},r}),\, s\in\R\}$.
  Then $\Om^+_{\bar{x},r}$ is symmetric with respect to $\Si\times\{0\}$, and
  $$\mathcal{H}^{n+1}(\Om^+_{\bar{x},r})\le 2r \mathcal{H}^n(\pi(M\cap \mathfrak{D}_{\bar{x},r}))\le\f12\hat{\be}r \mathcal{H}^n(B_r(x)).$$
Combining \eqref{Om+mfDxr}, we have
\begin{equation}\aligned\label{VolOm+Dxr}
\mathcal{H}^{n+1}\left(\Om^+\cap \mathfrak{D}_{\bar{x},r}\setminus\Om^+_{\bar{x},r}\right)\ge\f12\hat{\be}r\mathcal{H}^n(B_r(x)).
\endaligned
\end{equation}
For any $s\in\R$, we define a family of subsets $W_{\bar{x},r}(s)\subset B_r(x)$ by
$$\Om^+\cap \mathfrak{D}_{\bar{x},r}\setminus\Om^+_{\bar{x},r}=\bigcup_{s\in\R}  W_{\bar{x},r}(s)\times\{s\}.$$
Since $\Om^+\cap \mathfrak{D}_{\bar{x},r}\setminus\Om^+_{\bar{x},r}$ is symmetric with respect to $\Si\times\{0\}$, then
$W_{\bar{x},r}(s)=W_{\bar{x},r}(-s)$.
By the definition of $\Om^+_{\bar{x},r}$, we have
\begin{equation}\aligned
\mathcal{H}^{n+1}\left(\Om^+\cap \mathfrak{D}_{\bar{x},r}\setminus\Om^+_{\bar{x},r}\right)\le&\f14\hat{\be}r\mathcal{H}^n(W_{\bar{x},r}(0))+2r\mathcal{H}^n(W_{\bar{x},r}(\hat{\be}/8))\\
\le&\f14\hat{\be}r\mathcal{H}^n(B_{r}(x))+2r\mathcal{H}^n(W_{\bar{x},r}(\hat{\be}/8)).
\endaligned
\end{equation}
Combining \eqref{VolOm+Dxr}, we get
\begin{equation}\aligned
\mathcal{H}^n(W_{\bar{x},r}(-\hat{\be}/8))=\mathcal{H}^n(W_{\bar{x},r}(\hat{\be}/8))\ge\f18\hat{\be}\mathcal{H}^n(B_{r}(x)).
\endaligned
\end{equation}
Uniting with \eqref{lbdcylinderM} for $t\ge-\min\{\f12\tau,\f18\hat{\be}r,\be'r\}$, we complete the proof of the claim \eqref{U+low}.
  \item Case 2: $\mathcal{H}^n(\pi(M\cap \mathfrak{D}_{\bar{x},r}))>\f14\hat{\be} \mathcal{H}^n(B_r(x))$.
  From \eqref{HnEtDxr} and $\de<\f18\hat{\be}$, we have
$$\mathcal{H}^n(E_t)\le\de\mathcal{H}^n(B_r(x))\le\f18\hat{\be} \mathcal{H}^n(B_r(x)).$$
For $|t|\le\be'r$, we have
\begin{equation}\aligned
&\mathcal{H}^n(U^+_t)+\mathcal{H}^{n}\left(B_r(x)\setminus B_{\sqrt{1-(\be')^2}r}(x)\right)\ge\mathcal{H}^n(\pi(M\cap \mathfrak{D}_{\bar{x},r}\setminus E_t))\\
\ge&\mathcal{H}^n(\pi(M\cap \mathfrak{D}_{\bar{x},r}))-\mathcal{H}^n(E_t)\ge\f{\hat{\be}}4\mathcal{H}^n(B_r(x))-\f{\hat{\be}}8\mathcal{H}^n(B_r(x))=\f{\hat{\be}}{8}\mathcal{H}^n(B_r(x)),
\endaligned
\end{equation}
then combining \eqref{HnBr1-hbebe}, we have
\begin{equation}\aligned
\mathcal{H}^n(U^+_t)\ge&\f{\hat{\be}}{8}\mathcal{H}^n(B_r(x))-\mathcal{H}^{n}\left(B_r(x)\setminus B_{(1-\be')r}(x)\right)\\
\ge&\f{\hat{\be}}8\mathcal{H}^n(B_r(x))-\f{\hat{\be}}{16} \mathcal{H}^n(B_r(x))=\f{\hat{\be}}{16}\mathcal{H}^n(B_r(x)).
\endaligned
\end{equation}
Uniting with \eqref{lbdcylinderM} for $t\ge-\min\{\f12\tau,\f18\hat{\be}r,\be'r\}$, we complete the proof of the claim \eqref{U+low}.
\end{itemize}

Combining Lemma \ref{ISOPOIN} and \eqref{U+low}, for $t\ge-\min\{\f12\tau,\f18\hat{\be}r,\be'r\}$ we have
\begin{equation}\aligned\label{pElowx}
\mathcal{H}^{n-1}(\p E_t\cap \mathfrak{D}_{\bar{x},r+\tau})\ge\mathcal{H}^{n-1}(\p E_t\cap \mathfrak{D}_{(x,t),r-t})\ge\f{\de^*}r\left(\mathcal{H}^n(B_r(x))\right)^{\f{1}\a}\left(\mathcal{H}^n(U^-_t)\right)^{\f{\a-1}\a}
\endaligned
\end{equation}
for some positive constant $\de^*$ depending only on $C_D,C_N$.
Let $V_{x,t}$ be a domain enclosed by $\p \mathfrak{D}_{(x,t),r-t}^+$ and $M$ in $\mathfrak{D}_{(x,t),r-t}$ such that $\p V_{x,t}\cap \mathfrak{D}_{(x,t),r-t}^+\subset E_t\cup U^-_t$.
Since $M$ is area-minimizing in $B_R(p)\times\R$, then
\begin{equation}\aligned
\mathcal{H}^n(E_t\cap\mathfrak{D}_{\bar{x},r})=\mathcal{H}^n(\p V_{x,t}\cap M)\le\mathcal{H}^{n}(\p V_{x,t}\cap\p \mathfrak{D}_{(x,t),r-t})+\mathcal{H}^n(U^{-}_t)\le\left(\sqrt{2}+1\right)\mathcal{H}^n(U^{-}_t).
\endaligned
\end{equation}
Combining \eqref{pElowx}, we have
\begin{equation}\aligned
\mathcal{H}^{n-1}(\p E_t\cap \mathfrak{D}_{\bar{x},r+\tau})
\ge\left(\sqrt{2}+1\right)^{\f{1-\a}\a}\f{\de^*}r\left(\mathcal{H}^n(B_r(x))\right)^{\f{1}\a}\left(\mathcal{H}^n(E_t\cap\mathfrak{D}_{\bar{x},r})\right)^{\f{\a-1}\a}.
\endaligned
\end{equation}
With \eqref{UplowM}, we complete the proof.
\end{proof}

Moreover, we can get a different version of Lemma \ref{Sob*} as follows, which is a key ingredient of Neumann-Poincar$\mathrm{\acute{e}}$ inequality \eqref{B1pPITh*}.
\begin{lemma}\label{NPEt}
There is a constant $\th^*$ depending only on $C_D,C_N$ such that for any $\bar{x}=(x,u(x))$ with $B_{3r}(x)\subset B_R(p)$, for any $E_t=M\cap \mathfrak{D}_{\bar{x},3r}^t$ or $E_t=M\cap \mathfrak{D}_{\bar{x},3r}\setminus\mathfrak{D}_{\bar{x},3r}^t$ with $t\in(-r,r)$, we have
\begin{equation}\aligned
\left(\mathcal{H}^n(M\cap\mathfrak{D}_{\bar{x},r})\right)^{\f1\a}\left(\mathcal{H}^n(E_t\cap\mathfrak{D}_{\bar{x},r})\right)^{\f{\a-1}\a}\le \th^*r\mathcal{H}^{n-1}(\p E_t\cap \mathfrak{D}_{\bar{x},3r})
\endaligned
\end{equation}
provided $\mathcal{H}^n(E_t\cap\mathfrak{D}_{\bar{x},r})\le\f12\mathcal{H}^n(M\cap\mathfrak{D}_{\bar{x},r})$.
\end{lemma}
\begin{proof}
Without loss of generality, we can assume $u(x)=0$.
By considering $-u,E_{-t}$ instead of $u,E_t$, respectively, we only need to show the case $E_t=M\cap \mathfrak{D}_{\bar{x},3r}^t$ for any $t\in(-r,r)$.
Let $U^-_t$ be a subset of $B_{r+|t|}(x)\times\{t\}$ defined by
$$U^-_t=\Om^-\cap\left(B_{r+|t|}(x)\times\{t\}\right),$$
and $U^+_t=B_{r+|t|}(x)\times\{t\}\setminus U^-_t$.
Let $V_{x,t}$ be a domain enclosed by $\p\left(\mathfrak{D}_{(x,t),r+|t|}\setminus \mathfrak{D}_{(x,t),r+|t|}^+\right)$ and $M$ in $\mathfrak{D}_{(x,t),r+|t|}$ such that $\p V_{x,t}\cap \mathfrak{D}_{(x,t),r+|t|}\subset (M\setminus E_t)\cup U^+_t$.
Since $M$ is area-minimizing in $B_R(p)\times\R$, then
\begin{equation}\aligned
\mathcal{H}^n(\p V_{x,t}\cap M)\le\mathcal{H}^{n}(\p V_{x,t}\cap\p \mathfrak{D}_{(x,t),r+|t|})+\mathcal{H}^n(U^{+}_t)\le\left(\sqrt{2}+1\right)\mathcal{H}^n(U^{+}_t).
\endaligned
\end{equation}
Combining $\mathcal{H}^n(E_t\cap\mathfrak{D}_{\bar{x},r}))\le\f12\mathcal{H}^n(M\cap\mathfrak{D}_{\bar{x},r}))$, we have
\begin{equation}\aligned
\f12\mathcal{H}^n(M\cap\mathfrak{D}_{\bar{x},r}))\le\mathcal{H}^n(\mathfrak{D}_{\bar{x},r}\cap M\setminus E_t) \le\mathcal{H}^n(\p V_{x,t}\cap M)\le\left(\sqrt{2}+1\right)\mathcal{H}^n(U^{+}_t).
\endaligned
\end{equation}
With \eqref{UplowM}, we have
\begin{equation}\aligned
\mathcal{H}^n(U^+_t)\ge\f{\sqrt{2}-1}2\be_*\mathcal{H}^n(B_r(x)).
\endaligned
\end{equation}
Combining Lemma \ref{ISOPOIN} we have
\begin{equation}\aligned\label{pElowx*}
\mathcal{H}^{n-1}(\p E_t\cap \mathfrak{D}_{\bar{x},3r})\ge\mathcal{H}^{n-1}(\p E_t\cap \mathfrak{D}_{(x,t),r+|t|})\ge\f{\ep_*}r\left(\mathcal{H}^n(B_r(x))\right)^{\f{1}\a}\left(\mathcal{H}^n(U^-_t)\right)^{\f{\a-1}\a}
\endaligned
\end{equation}
for some positive constant $\ep_*$ depending only on $C_D,C_N$.
Let $W_{x,t}$ be a domain enclosed by $\p \mathfrak{D}_{(x,t),r+|t|}^+$ and $M$ in $\mathfrak{D}_{(x,t),r+|t|}$ such that $\p W_{x,t}\cap \mathfrak{D}_{(x,t),r+|t|}^+\subset E_t\cup U^-_t$.
Area-minimizing $M$ in $B_R(p)\times\R$ implies
\begin{equation}\aligned
\mathcal{H}^n(E_t\cap\mathfrak{D}_{\bar{x},r})\le\mathcal{H}^{n}(\p W_{x,t}\cap\p \mathfrak{D}_{(x,t),r+|t|})+\mathcal{H}^n(U^{-}_t)\le\left(\sqrt{2}+1\right)\mathcal{H}^n(U^{-}_t).
\endaligned
\end{equation}
Combining \eqref{pElowx*}, we have
\begin{equation}\aligned
\mathcal{H}^{n-1}(\p E_t\cap \mathfrak{D}_{\bar{x},3r})
\ge\left(\sqrt{2}+1\right)^{\f{1-\a}\a}\f{\ep_*}r\left(\mathcal{H}^n(B_r(x))\right)^{\f{1}\a}\left(\mathcal{H}^n(E_t\cap\mathfrak{D}_{\bar{x},r})\right)^{\f{\a-1}\a}.
\endaligned
\end{equation}
With \eqref{UplowM}, we complete the proof.
\end{proof}

Assume $u>0$ on $B_R(p)$.
Let $\phi$ be a monotonic increasing or monotonic decreasing $C^1$-function on $\R^+$, and $\Phi(x)=\phi(u(x))$ on $B_R(p)\subset\Si$.
For any $\bar{x}\in\Si\times\R$, and $r>0$, we put
$$\mathscr{B}_{r}(\bar{x})=M\cap \mathfrak{D}_{\bar{x},r}.$$
Let $\na$ be the Levi-Civita connection of $M$ with respect to its induced metric from $\Si\times\R$.
Let us prove Sobolev inequalities for $\Phi$ on $M$ using Lemma \ref{Sob*}.
\begin{lemma}\label{SobMG}
Suppose $\Phi>0$ on $B_R(p)$.
There are
\begin{equation}\aligned\label{P11111}
\left(\mathcal{H}^n(\mathscr{B}_{r}(\bar{x}))\right)^{\f1\a}\left(\int_{\mathscr{B}_{r}(\bar{x})}\Phi^{\f \a{\a-1}}\right)^{\f{\a-1}\a}\le\th r\left(\int_{\mathscr{B}_{r+\tau}(\bar{x})}|\na\Phi|+\f1{\tau}\int_{\mathscr{B}_{r}(\bar{x})}\Phi\right),
\endaligned
\end{equation}
and
\begin{equation}\aligned\label{P22222}
\left(\mathcal{H}^n(\mathscr{B}_{r}(\bar{x}))\right)^{\f1\a}\left(\int_{\mathscr{B}_{r}(\bar{x})}\Phi^{\f{2\a}{\a-1}}\right)^{\f{\a-1}\a}
\le\th \left(r^2\int_{\mathscr{B}_{r+\tau}(\bar{x})}|\na\Phi|^2+\f{2r}{\tau}\int_{\mathscr{B}_{r+\tau}(\bar{x})}\Phi^2\right)
\endaligned
\end{equation}
for all $r\ge\tau>0$ and $B_{2r}(x)\subset B_R(p)$ with $\bar{x}=(x,u(x))\in M$, where $\th$ is the constant depending only on $C_D,C_N$ defined in Lemma \ref{Sob*}.
\end{lemma}
\begin{proof}
We fix $r>0$ and $\tau\in(0,r]$.
For any $t\in\R$, we define an open set in $M$ by
$$E_t=\{x\in M\cap \mathfrak{D}_{\bar{x},r+\tau}|\, \Phi(x)>t\}.$$
From the monotonicity of $\phi$, there exists a number $s_{t}\in\R$ with $\phi(s_{t})=t$ such that
$$E_{t}=\{(q,u(q))\in M\cap \mathfrak{D}_{\bar{x},r+\tau}|\, u(q)>s_{t}\}$$
or
$$E_{t}=\{(q,u(q))\in M\cap \mathfrak{D}_{\bar{x},r+\tau}|\, u(q)<s_{t}\}.$$
From Lemma \ref{Sob*},
\begin{equation}\aligned\label{ISOEtpEt}
\left(\mathcal{H}^n(\mathscr{B}_r(x))\right)^{\f1\a}\left(\mathcal{H}^n(E_{t}\cap\mathfrak{D}_{\bar{x},r})\right)^{\f{\a-1}\a}
\le \th\left(r\mathcal{H}^{n-1}(\p E_{t}\cap \mathfrak{D}_{\bar{x},r+\tau})+\f r{\tau}\mathcal{H}^n(E_{t}\cap\mathfrak{D}_{\bar{x},r}))\right).
\endaligned
\end{equation}
for the constant $\th$ depending only on $C_D,C_N$.

By co-area formula,
\begin{equation}\aligned
\int_{M\cap \mathfrak{D}_{\bar{x},r+\tau}}|\na\Phi|=\int_0^\infty\mathcal{H}^{n-1}(\p E_t\cap \mathfrak{D}_{\bar{x},r+\tau})dt.
\endaligned
\end{equation}
By Fubini's theorem, the $(n+1)$-dimensional Hausdorff measure of $\{(x,t)\in M\cap \mathfrak{D}_{\bar{x},r}|\ 0<t<\Phi(x)\}$ is equal to
\begin{equation}\aligned
\int_{M\cap \mathfrak{D}_{\bar{x},r}}\Phi=\int_0^\infty\mathcal{H}^n(E_t\cap \mathfrak{D}_{\bar{x},r})dt.
\endaligned
\end{equation}
Then from \eqref{ISOEtpEt} we have
\begin{equation}\aligned
\int_{M\cap \mathfrak{D}_{\bar{x},r+\tau}}|\na\Phi|+\f1{\tau}\int_{M\cap \mathfrak{D}_{\bar{x},r}}\Phi=&\int_0^\infty\mathcal{H}^{n-1}(\p E_t\cap \mathfrak{D}_{\bar{x},r+\tau})dt+\f1{\tau}\int_0^\infty\mathcal{H}^n(E_t\cap \mathfrak{D}_{\bar{x},r})dt\\
\ge&\f{\left(\mathcal{H}^n(\mathscr{B}_r(\bar{x}))\right)^{\f1\a}}{\th r}\int_0^{\infty}\left(\mathcal{H}^n(E_t\cap \mathfrak{D}_{\bar{x},r})\right)^{\f{\a-1}\a}dt.
\endaligned
\end{equation}
From a result of Hardy-Littlewood-P$\mathrm{\acute{o}}$lya (see also the proof of co-area formula in \cite{SY}), one gets
\begin{equation}\aligned
\int_{M\cap \mathfrak{D}_{\bar{x},r+\tau}}|\na\Phi|+\f1{\tau}\int_{M\cap \mathfrak{D}_{\bar{x},r}}\Phi
\ge&\f{\left(\mathcal{H}^n(\mathscr{B}_r(\bar{x}))\right)^{\f1\a}}{\th r}\left(\f \a{\a-1}\int_0^\infty t^\f1{\a-1}\mathcal{H}^{n}(E_t\cap \mathfrak{D}_{\bar{x},r})dt\right)^{\f{\a-1}\a}\\
=&\f{\left(\mathcal{H}^n(\mathscr{B}_r(\bar{x}))\right)^{\f1\a}}{\th r}\left(\int_0^\infty \mathcal{H}^{n}\left(E_{s^{\f{\a-1}\a}}\cap \mathfrak{D}_{\bar{x},r}\right)ds\right)^{\f{\a-1}\a}\\
=&\f{\left(\mathcal{H}^n(\mathscr{B}_r(\bar{x}))\right)^{\f1\a}}{\th r}\left(\int_{M\cap \mathfrak{D}_{\bar{x},r}}\Phi^{\f \a{\a-1}}d\mu\right)^{\f{\a-1}\a}.
\endaligned
\end{equation}
This completes the proof of \eqref{P11111}.

Note that $\phi^2$ is still a monotonic increasing or monotonic decreasing $C^1$-function on $\R^+$.
For any $r\ge\tau$ and $B_{r+\tau}(x)\subset B_R(p)$ with $(x,u(x))\in M$, from \eqref{P11111} and Cauchy inequality we have
\begin{equation}\aligned
\left(\mathcal{H}^n(\mathscr{B}_{r}(\bar{x}))\right)^{\f1\a}\left(\int_{\mathscr{B}_{r}(\bar{x})}\Phi^{\f{2\a}{\a-1}}\right)^{\f{\a-1}\a}
&\le\th r\left(2\int_{\mathscr{B}_{r+\tau}(\bar{x})}\Phi|\na\Phi|+\f1{\tau}\int_{\mathscr{B}_{r}(\bar{x})}\Phi^2\right)\\
\le&\th \left(r^2\int_{\mathscr{B}_{r+\tau}(\bar{x})}|\na\Phi|^2+\f{2r}{\tau}\int_{\mathscr{B}_{r+\tau}(\bar{x})}\Phi^2\right).
\endaligned
\end{equation}
This completes the proof of \eqref{P22222}.
\end{proof}

Let us show the following Neumann-Poincar$\mathrm{\acute{e}}$ inequality using Lemma \ref{NPEt}.
\begin{lemma}\label{NPTh*}
Let $\th^*$ be the constant in Lemma \ref{NPEt}. Then
\begin{equation}\aligned\label{B1pPITh*}
\int_{\mathscr{B}_{r}(\bar{x})}|\Phi-\bar{\Phi}_{\bar{x},r}|\le 2\th^* r\int_{\mathscr{B}_{3r}(\bar{x})}|\na \Phi|
\endaligned
\end{equation}
for all $B_{3r}(x)\subset B_{R}(p)$ with $(x,u(x))\in M$,
where $\bar{\Phi}_{x,r}$ is the mean of $\Phi$ on $\mathscr{B}_{r}(\bar{x})$, i.e., $\bar{\Phi}_{\bar{x},r}=\fint_{\mathscr{B}_{r}(\bar{x})}\Phi$.
\end{lemma}
\begin{proof}
Let $\bar{\Phi}_{\bar{x},r}$ be the mean of $\Phi$ on $\mathscr{B}_{r}(\bar{x})$, i.e.,
$$\bar{\Phi}_{\bar{x},r}=\fint_{\mathscr{B}_{r}(\bar{x})}\Phi=\f1{\mathcal{H}^n(\mathscr{B}_{r}(\bar{x}))}\int_{\mathscr{B}_{r}(\bar{x})}\Phi.$$
For any fixed $\bar{x}=(x,u(x))\in M$, $r>0$ with $B_{3r}(x)\subset B_R(p)$, let
$$U^+_{s,t}=\{y\in \mathscr{B}_{s}(\bar{x})|\, \Phi(y)>\bar{\Phi}_{\bar{x},r}+t\},$$
and
$$U^-_{s,t}=\{y\in \mathscr{B}_{s}(\bar{x})|\, \Phi(y)<\bar{\Phi}_{\bar{x},r}+t\}$$
for all $s\in(0,3r]$ and $t\in\R$.
Without loss of generality, we assume $\mathcal{H}^n(U^+_{r,0})\le\mathcal{H}^n(U^-_{r,0})$. Then
$$\mathcal{H}^n(U^+_{r,t})\le\mathcal{H}^n(U^+_{r,0})\le\mathcal{H}^n(U^-_{r,0})\le\mathcal{H}^n(U^-_{r,t})$$
for any $t\ge0$.
In particular,
\begin{equation}\aligned
\mathcal{H}^n(U^-_{r,t})\ge\f12\mathcal{H}^n(\mathscr{B}_{r}(\bar{x})).
\endaligned
\end{equation}

From Lemma \ref{NPEt}, there holds
\begin{equation}\aligned
\mathcal{H}^{n-1}\left(\p U^+_{3r,t}\cap \mathscr{B}_{3r}(\bar{x})\right)\ge\f1{\th^* r}\left(\mathcal{H}^n(\mathscr{B}_{r}(\bar{x}))\right)^{\f{1}\a}\left(\mathcal{H}^n(U^+_{r,t})\right)^{\f{\a-1}\a}\ge\f1{\th^* r}\mathcal{H}^n(U^+_{r,t}),
\endaligned
\end{equation}
where $\th^*$ is the constant in Lemma \ref{NPEt}.
Combining co-area formula, we have
\begin{equation}\aligned
&\int_{U^+_{r,0}}\left(\Phi-\bar{\Phi}_{\bar{x},r}\right)=\int_0^\infty\mathcal{H}^n(U^+_{r,t})dt\\
\le&\th^* r\int_0^\infty\mathcal{H}^{n-1}\left(\p U^+_{3r,t}\cap \mathscr{B}_{3r}(\bar{x})\right)dt\le \th^* r\int_{\mathscr{B}_{3r}(\bar{x})}|\na\Phi|.
\endaligned
\end{equation}
Combining the definition of $\bar{\Phi}_{\bar{x},r}$, we get
\begin{equation}\aligned\label{fbfxrdf}
&\int_{\mathscr{B}_{r}(\bar{x})}\left|\Phi-\bar{\Phi}_{\bar{x},r}\right|=\int_{U^+_{r,0}}\left(\Phi-\bar{\Phi}_{\bar{x},r}\right)-\int_{U^-_{r,0}}\left(\Phi-\bar{\Phi}_{\bar{x},r}\right)\\
=&2\int_{U^+_{r,0}}\left(\Phi-\bar{\Phi}_{\bar{x},r}\right)\le 2\th^* r\int_{\mathscr{B}_{3r}(\bar{x})}|\na\Phi|.
\endaligned
\end{equation}
This completes the proof.
\end{proof}
Remark. From the proof of Lemma \ref{NPTh*}, clearly the inequality \eqref{B1pPITh*} holds without the monotonicity of $\phi$.

\section{Harnack's inequality for minimal graphic functions}

For $R>0$,
let $M$ be a minimal graph over $B_{4R}(p)\subset\Si$ with the graphic function $u$.
We always assume that the minimal graphic function $u$ is not a constant and $u>0$ on $B_{4R}(p)$.
Let $\tilde{u}(z)=u(x)$ for any $z=(x,u(x))\in M$, and we usually denote $\tilde{u}$ by $u$, which will not cause confusion from the context in general.
Let $\De$ be the Laplacian on $M$ with respect to its induced metric from $\Si\times\R$. Then $u$ is harmonic on $M$ (see also (2.2) in \cite{DJX1}), i.e.,
\begin{equation}\aligned\label{Deu=0}
\De u=0.
\endaligned
\end{equation}

Denote $\bar{p}=(p,u(p))\in\Si\times\R$. Recall $\mathscr{B}_{r}(\bar{p})=M\cap \mathfrak{D}_{\bar{p},r}$ and
$$\mathfrak{D}_{\bar{p},r}=\{(y,s)\in\Si\times\R|\, d(p,y)+|s-u(p)|<r\}.$$
For any function $\psi\in L^k(\mathscr{B}_r(\bar{p}))$ with each $k>0$ and $r\in(0,4R)$, we set
$$||\psi||_{k,r}=\left(\f1{\mathcal{H}^n(\mathscr{B}_r(\bar{p}))}\int_{\mathscr{B}_r(\bar{p})}|\psi|^kd\mu\right)^{1/k},$$
where $d\mu$ is the volume element of $M$.
Let $\r_{\bar{p}}$ be a Lipschitz function on $\Si\times\R$ defined by
$$\r_{\bar{p}}(\bar{x})=d(x,p)+|t-u(p)|$$
for any $\bar{x}=(x,t)$. 
Then $\mathfrak{D}_{\bar{p},r}=\{z\in\Si\times\R|\, \r_{\bar{p}}(z)<r\}$.
Let $\na$ and $\bn$ be the Levi-Civita connections of $M$ and $\Si\times\R$, respectively.
Then
\begin{equation}\aligned\label{nabarr}
|\na\r_{\bar{p}}|\le|\bn\r_{\bar{p}}|\le\sqrt{2}.
\endaligned
\end{equation}

Let $\phi$ be a monotonic increasing or monotonic decreasing positive $C^1$-function on $\R^+$, and $\Phi(x)=\phi(u(x))$ on $B_{4R}(p)\subset\Si$.
Now, let us carry out De Giorgi-Nash-Moser iteration for getting the Harnack' inequality of $u$ with the help of
the Sobolev inequality and the Neumann-Poincar$\mathrm{\acute{e}}$ inequality for the function $\Phi$.
\begin{lemma}
Suppose $\De\Phi\ge0$ on $M$. For any $r\in(0,2R]$, there is a constant $c_0$ depending only on $C_D,C_N$ such that
\begin{equation}\aligned\label{subharmkge2}
||\Phi||_{\infty,\de r}\le c_0(1-\de)^{-\f {2\a}k}||\Phi||_{k,r}
\endaligned
\end{equation}
for any $k\ge2$ and $\de\in(0,1)$.
\end{lemma}
\begin{proof}
For any constant $\ell\ge1$ and any Lipschitz function $\e$ with supp$\e\subset \mathscr{B}_{4R}(\bar{p})$, from $\De\Phi\ge0$ we have
\begin{equation}\aligned
0\ge&-\int \Phi^{2\ell-1}\e^2\De \Phi=(2\ell-1)\int \Phi^{2\ell-2}\e^2|\na \Phi|^2+2\int \Phi^{2\ell-1}\e\na \Phi\cdot\na\e\\
\ge&(2\ell-1)\int \Phi^{2\ell-2}\e^2|\na \Phi|^2-\f{\ell}2\int \Phi^{2\ell-2}\e^2|\na \Phi|^2-\f2{\ell}\int \Phi^{2\ell}|\na\e|^2\\
=&\left(\f32\ell-1\right)\int \Phi^{2\ell-2}\e^2|\na \Phi|^2-\f2{\ell}\int \Phi^{2\ell}|\na\e|^2,
\endaligned
\end{equation}
which infers
\begin{equation}\aligned\label{naellle2}
\int\left|\na \Phi^\ell\right|^2\e^2\le\f{4\ell}{3\ell-2}\int \Phi^{2\ell}|\na\e|^2\le4\int \Phi^{2\ell}|\na\e|^2.
\endaligned
\end{equation}
For each $r,\tau>0$ with $\tau\le r$, let $\e$ be a Lipschitz function defined by $\e=1$ on $\mathscr{B}_{r+\f\tau2}(\bar{p})$, $\e=\f2\tau\left(r+\tau-\r_{\bar{p}}\right)$ on $\mathscr{B}_{r+\tau}(\bar{p})\setminus\mathscr{B}_{r+\f\tau2}(\bar{p})$,
$\e=0$ outside $\mathscr{B}_{r+\tau}(\bar{p})$.
Then $|\na\e|\le2\sqrt{2}/\tau$ from \eqref{nabarr}.
Combining Lemma \ref{SobMG} and \eqref{naellle2}, we have
\begin{equation}\aligned\nonumber
||\Phi^{2\ell}||_{\f {\a}{\a-1},r}\le& \th\left(r^2||\na \Phi^{\ell}||_{2,r+\f\tau2}^2+\f{4r}{\tau}||\Phi^{2\ell}||_{1,r+\f\tau2}\right)
\le \th\left(4r^2\int \Phi^{2\ell}|\na\e|^2+\f{4r}{\tau}||\Phi^{2\ell}||_{1,r+\tau}\right)\\
\le& \th\left(32\f{r^2}{\tau^2}||\Phi^{2\ell}||_{1,r+\tau}+\f{4r}{\tau}||\Phi^{2\ell}||_{1,r+\tau}\right)\le c_\th \f{r^2}{\tau^2}||\Phi^{2\ell}||_{1,r+\tau}
\endaligned
\end{equation}
with $c_\th=36\th$, where $\th$ is the constant depending only on $C_D,C_N$ defined in Lemma \ref{Sob*}.
Then one has
\begin{equation}\aligned\label{ite}
||\Phi||_{\f {2\a\ell}{\a-1},r}\le c_\th^{\f1{2\ell}}r^{\f1{\ell}}\tau^{-\f1{\ell}}||\Phi||_{2\ell,r+\tau}.
\endaligned
\end{equation}

For any $\de\in(0,1)$, $k\ge2$ and any integer $i\ge-1$, set $\ell_i=\f k2\left(\f {\a}{\a-1}\right)^i$, $\tau_i=2^{-(1+i)}(1-\de)r$ and $r_i=r-\sum_{j=0}^i\tau_j=\de r+\tau_i\le r$.
By iterating \eqref{ite}, for $i\ge0$ we have
\begin{equation}\aligned\label{itei}
||\Phi||_{\f {2\a}{\a-1}\ell_{i},r_i}\le c_\th^{\f1{2\ell_i}}r_i^{\f1{\ell_i}}\tau_i^{-\f1{\ell_i}}||\Phi||_{\f {2\a}{\a-1}\ell_{i-1},r_{i-1}}\le\prod_{j=0}^i c_\th^{\f1{2\ell_j}}r_j^{\f1{\ell_j}}\tau_j^{-\f1{\ell_j}}||\Phi||_{k,r}.
\endaligned
\end{equation}
Note that $\tau_j/r_j\ge2^{-(1+j)}(1-\de)$ for every $j\ge0$.
Letting $i\rightarrow\infty$, then
\begin{equation}\aligned\label{Phiinftyder}
||\Phi||_{\infty,\de r}\le&\prod_{j=0}^\infty c_\th^{\f1{2\ell_j}}\left(\f{2^{1+j}}{(1-\de)}\right)^{\f1{\ell_j}}||\Phi||_{k,r}
=c_\th^{\sum_{j=0}^\infty\f1{2\ell_j}}2^{\sum_{j=0}^\infty\f{1+j}{\ell_j}}(1-\de)^{-\sum_{j=0}^\infty\f1{\ell_j}}||\Phi||_{k,r}.
\endaligned
\end{equation}
By the definition of $\ell_j$, it follows that
\begin{equation}\aligned
\sum_{j=0}^\infty\f1{2\ell_j}=&\f1k\sum_{j=0}^\infty\left(\f {\a-1}\a\right)^{j}=\f \a{k},\\
\sum_{j=0}^\infty\f{j+1}{\ell_j}=&\f2k\sum_{j=0}^\infty (j+1)\left(\f {\a-1}\a\right)^{j}=\f{2\a^2}{k}.
\endaligned
\end{equation}
Hence from \eqref{Phiinftyder} we get
\begin{equation}\aligned
||\Phi||_{\infty,\de r}\le 2^{\f{2\a^2}k}c_\th^{\f{\a}k}(1-\de)^{-\f{2\a}k}||\Phi||_{k,r}.
\endaligned
\end{equation}
This completes the proof.
\end{proof}

\begin{theorem}\label{MVinequR}
Suppose $\De\Phi\ge0$ on $M$.
For any $k>0$, there is a constant $c_{k}$ depending only on $k,C_D,C_N$ such that for any $r\in(0,2R]$ one has
\begin{equation}\aligned\label{MVinequR*}
\sup_{\mathscr{B}_{\de r}(\bar{p})}\Phi\le c_{k}(1-\de)^{-\f{2\a}k}||\Phi||_{k,r}.
\endaligned
\end{equation}
\end{theorem}
\begin{proof}
From \eqref{subharmkge2}, we only need to show \eqref{MVinequR*} for $0<k<2$.
From \eqref{subharmkge2}, we have
\begin{equation}\aligned\label{Phigenerate}
||\Phi||_{\infty,\de r}\le& c_{0}(1-\de)^{-\a}||\Phi||_{2,r}.
\endaligned
\end{equation}
Put $r_0=\de r$, $r_i=\de r+\sum_{j=1}^i2^{-j}(1-\de)r=r-2^{-i}(1-\de)r$, and $\de_i=r_{i-1}/r_i$. Then
$$1-\de_i=\f{r_i-r_{i-1}}{r_i}=\f{2^{-i}(1-\de)r}{r-2^{-i}(1-\de)r}\ge2^{-i}(1-\de).$$
Note that $\f{1+\de}2r\le r_i\le r$ for all $i\ge1$. Then from \eqref{VG}\eqref{UplowM} we have
$$\mathcal{H}^n(\mathscr{B}_r(\bar{p}))\le\sqrt{2}\mathcal{H}^{n}(B_r(p))\le 2^{\a+\f12}C_D\mathcal{H}^{n}\left(B_{\f r2}(p)\right)\le 2^{\a+\f12}C_D\be_*^{-1}\mathcal{H}^{n}(\mathscr{B}_{r_i}(\bar{p})).$$
From \eqref{Phigenerate}, for $i\ge1$ we have
\begin{equation}\aligned
||\Phi||_{\infty,r_{i-1}}\le& c_0(1-\de_i)^{-\a}||\Phi||_{2,r_i}\le c_0(1-\de_i)^{-\a}||\Phi||_{k,r_i}^{\f k2}||\Phi||_{\infty,r_i}^{1-\f k2}\\
\le& c_02^{i\a}(1-\de)^{-\a}\left(2^{\a+\f12}C_D\be_*^{-1}\right)^{\f12}||\Phi||_{k,r}^{\f k2}||\Phi||_{\infty,r_i}^{1-\f k2}.
\endaligned
\end{equation}
Set $\widetilde{c_0}=c_0\left(2^{\a+\f12}C_D\be_*^{-1}\right)^{\f12}$, which is a constant depending only on $C_D,C_N$.
Iterating the above inequality implies
\begin{equation}\aligned\label{Phiinftyr0}
||\Phi||_{\infty,r_0}\le& \widetilde{c_0}2^{\a}(1-\de)^{-\a}||\Phi||_{k,r}^{\f k2}||\Phi||_{\infty,r_1}^{1-\f k2}\\
\le& \prod_{j=0}^i\left(\widetilde{c_0}2^{(j+1)\a}(1-\de)^{-\a}||\Phi||_{k,r}^{\f k2}\right)^{(1-k/2)^j}||\Phi||_{\infty,r_{i+1}}^{(1-k/2)^{i+1}}.
\endaligned
\end{equation}
By a direct computation, one has
\begin{equation}\aligned
\sum_{j=0}^\infty& \left(1-\f k2\right)^j=\f{2}{k},\\
\sum_{j=0}^\infty& (j+1)\left(1-\f k2\right)^j=\f{4}{k^2}.
\endaligned
\end{equation}
Letting $i\rightarrow\infty$ in \eqref{Phiinftyr0} infers
\begin{equation}\aligned
||\Phi||_{\infty,\de r}\le& \left(\prod_{j=0}^{\infty}\left(\widetilde{c_0}2^{(j+1)\a}(1-\de)^{-\a}\right)^{(1-k/2)^j}\right)||\Phi||_{k,r}^{\sum_{j=0}^{\infty}\f k2(1-k/2)^j}\\
=&\left(\prod_{j=0}^{\infty}\left(\widetilde{c_0}(1-\de)^{-\a}\right)^{(1-k/2)^j}\right)2^{\a\sum_{j=0}^{\infty}(j+1)(1-k/2)^j}||\Phi||_{k,r}^{\sum_{j=0}^{\infty}\f k2(1-k/2)^j}\\
=&2^{\f{4\a}{k^2}}\widetilde{c_0}^{\f2k}(1-\de)^{-\f{2\a}{k}}||\Phi||_{k,r}.
\endaligned
\end{equation}
This completes the proof.
\end{proof}

\begin{theorem}\label{Harnack}
Suppose $u>0$ on $B_{4R}(p)$. Then $u$ satisfies Harnack's inequality on $\mathscr{B}_{2R}(\bar{p})$:
\begin{equation}\aligned\label{Harnack}
\sup_{\mathscr{B}_{2R}(\bar{p})}u\le\vartheta\inf_{\mathscr{B}_{2R}(\bar{p})}u
\endaligned
\end{equation}
for some constant $\vartheta$ depending only on $C_D,C_N$.
\end{theorem}
\begin{proof}
Let
$$w=\log u-\fint_{\mathscr{B}_{r}(\bar{p})}\log u,$$
then from \eqref{Deu=0} one has
\begin{equation}\aligned\label{Dew}
\De w=-|\na w|^2.
\endaligned
\end{equation}
Let $\e$ be a Lipschitz function with compact support in $M\cap B_{4R}(p)$.
From \eqref{Dew}, for any $q\ge0$ integrating by parts implies
\begin{equation}\aligned
\int |\na w|^2\e^2|w|^q=&-\int\e^2|w|^q\De w=2\int\e|w|^q\na\e\cdot\na w+q\int\e^2|w|^{q-2}w|\na w|^2\\
\le&\f12\int |\na w|^2\e^2|w|^q+2\int|\na\e|^2|w|^q+q\int\e^2|w|^{q-1}|\na w|^2.
\endaligned
\end{equation}
Then
\begin{equation}\aligned\label{ewDwq}
\int \e^2|w|^q|\na w|^2\le4\int|\na\e|^2|w|^q+2q\int\e^2|w|^{q-1}|\na w|^2.
\endaligned
\end{equation}
For any $r\in(0,R]$, if we choose $\e_0=1$ on $\mathscr{B}_{3r}(\bar{p})$, $\e_0=\f {4r-\r_{\bar{p}}}r$ on $\mathscr{B}_{4r}(\bar{p})\setminus\mathscr{B}_{3r}(\bar{p})$,
$\e_0=0$ on $M\setminus\mathscr{B}_{4r}(\bar{p})$.
Then combining \eqref{nabarr} we get $|\na\e_0|\le\sqrt{2}/r$.
Choosing $q=0$ in \eqref{ewDwq}, we have
\begin{equation}\aligned\label{|naw|2}
\int_{\mathscr{B}_{3r}(\bar{p})} |\na w|^2\le4\int|\na\e_0|^2\le\f{8}{r^2}\mathcal{H}^{n}(\mathscr{B}_{4r}(\bar{p})).
\endaligned
\end{equation}
Combining the Neumann-Poincar$\mathrm{\acute{e}}$ inequality \eqref{B1pPITh*} for $w$, we have
\begin{equation}\aligned\label{Br|w|}
&\int_{\mathscr{B}_r(\bar{p})} |w|\le 2\th^*r\int_{\mathscr{B}_{3r}(\bar{p})} |\na w|\\
\le& 2\th^*r\left(\mathcal{H}^n(\mathscr{B}_{3r}(\bar{p}))\right)^{\f12}\left(\int_{\mathscr{B}_{3r}(\bar{p})} |\na w|^2\right)^{\f12}\le 4\sqrt{2}\th^*\mathcal{H}^{n}(\mathscr{B}_{4r}(\bar{p})).
\endaligned
\end{equation}
We choose $\e_1=1$ on $\mathscr{B}_{\f34r}(\bar{p})$, $\e_1=4\f {r-\r_{\bar{p}}}r$ on $\mathscr{B}_{r}(\bar{p})\setminus\mathscr{B}_{\f34r}(\bar{p})$,
$\e_1=0$ on $M\setminus\mathscr{B}_{r}(\bar{p})$.
Then we get $|\na\e_1|\le4\sqrt{2}/r$ from \eqref{nabarr}.
Without loss of generality, we may assume $\th^*\ge1$.
Choosing $q=1$ in \eqref{ewDwq}, combining \eqref{|naw|2}\eqref{Br|w|} we have
\begin{equation}\aligned\label{w|naw|2}
\int_{\mathscr{B}_{\f34r}(\bar{p})}|w||\na w|^2\le\f{2^7}{r^2}\int_{\mathscr{B}_r(\bar{p})}|w|+2\int_{\mathscr{B}_r(\bar{p})}|\na w|^2\le\f{2^{10}}{r^2}\th^*\mathcal{H}^{n}(\mathscr{B}_{4r}(\bar{p})).
\endaligned
\end{equation}

Let $r_j=\f12(1+2^{-j})r$ for each integer $j\ge0$.
Let $\e$ be the cut-off function on $\mathscr{B}_{r_{j}}(\bar{p})$ such that $\e=1$ on $\mathscr{B}_{r_{j+1}}(\bar{p})$, $\e=\f {r_{j}-\r_{\bar{p}}}{r_{j}-r_{j+1}}$ on $\mathscr{B}_{r_{j}}(\bar{p})\setminus\mathscr{B}_{r_{j+1}}(\bar{p})$,
$\e=0$ on $M\setminus\mathscr{B}_{r_j}(\bar{p})$.
Then combining \eqref{nabarr} we get $|\na\e|\le 2^{j+\f52}/r$.
From \eqref{ewDwq}, for any number $q\ge0$ and any integer $j\ge0$ we have
\begin{equation}\aligned\label{Brj-1e2wqnew2}
\int_{\mathscr{B}_{r_{j}}(\bar{p})}\e^2|w|^q|\na w|^2\le  \f{2^{2j+7}}{r^2}\int_{\mathscr{B}_{r_{j}}(\bar{p})}|w|^q+2q\int_{\mathscr{B}_{r_{j}}(\bar{p})}\e^2|w|^{q-1}|\na w|^2.
\endaligned
\end{equation}
From Young's inequality, one has
\begin{equation}\aligned\label{Young}
2q|w|^{q-1}\le\f12|w|^q+2^{2q-1}(q-1)^{q-1}\ \quad \mathrm{for}\ \ q\ge1,\\
|w|^q\le q|w|+(1-q)\ \quad \mathrm{for}\ \ q\in[0,1).
\endaligned
\end{equation}
Here, we let $0^0=1$ for the case $q=1$.
Then for $j\ge0$ and $q\ge1$, substituting \eqref{Young} into \eqref{Brj-1e2wqnew2} infers
\begin{equation}\aligned
\f12\int_{\mathscr{B}_{r_{j+1}}(\bar{p})}|w|^q|\na w|^2\le  &\f{2^{2j+7}}{r^2}\int_{\mathscr{B}_{r_{j}}(\bar{p})}|w|^q+2^{2q-1}(q-1)^{q-1}\int_{\mathscr{B}_{r_{j}}(\bar{p})}|\na w|^2\\
\le&  \f{2^{2j+7}}{r^2}\int_{\mathscr{B}_{r_{j}}(\bar{p})}|w|^q+2^{2q+2}(q-1)^{q-1}r^{-2}\mathcal{H}^{n}(\mathscr{B}_{4r}(\bar{p})),
\endaligned
\end{equation}
where we have used \eqref{|naw|2} in the above inequality. Combining Cauchy inequality, we have
\begin{equation}\aligned\label{Brjwqnaw}
&\int_{\mathscr{B}_{r_{j+1}}(\bar{p})}|w|^{q}|\na w|\le \f{r}{2^{j+5}}\int_{\mathscr{B}_{r_{j+1}}(\bar{p})}|w|^{q}|\na w|^2+\f{2^{j+3}}{r}\int_{\mathscr{B}_{r_{j+1}}(\bar{p})}|w|^{q}\\
\le&\f{2^{j+4}}{r}\int_{\mathscr{B}_{r_{j}}(\bar{p})}|w|^{q}+2^{2q-j-2}(q-1)^{q-1}r^{-1}\mathcal{H}^{n}(\mathscr{B}_{4r}(\bar{p}))
\endaligned
\end{equation}
for $q\ge1$ and $j\ge0$.
Moreover, for $j\ge0$ and $0\le q<1$, combining \eqref{|naw|2}\eqref{w|naw|2}\eqref{Young} one has
\begin{equation}\aligned
&\int_{\mathscr{B}_{r_{j+1}}(\bar{p})}|w|^q|\na w|^2\le q\int_{\mathscr{B}_{r_{1}}(\bar{p})}|w||\na w|^2+(1-q)\int_{\mathscr{B}_{r_{1}}(\bar{p})}|\na w|^2\\
\le&q\f{2^{10}}{r^2}\th^*\mathcal{H}^{n}(\mathscr{B}_{4r}(\bar{p}))+(1-q)\f{8}{r^2}\mathcal{H}^{n}(\mathscr{B}_{4r}(\bar{p}))
\le\f{2^{10}}{r^2}\th^*\mathcal{H}^{n}(\mathscr{B}_{4r}(\bar{p})).
\endaligned
\end{equation}
Then with Cauchy inequality, we have
\begin{equation}\aligned\label{Brjwqnaw*}
&\int_{\mathscr{B}_{r_{j+1}}(\bar{p})}|w|^{q}|\na w|\le \f{r}{2^{j+6}}\int_{\mathscr{B}_{r_{j+1}}(\bar{p})}|w|^{q}|\na w|^2+\f{2^{j+4}}{r}\int_{\mathscr{B}_{r_{j+1}}(\bar{p})}|w|^{q}\\
\le&\f{2^{j+4}}{r}\int_{\mathscr{B}_{r_{j+1}}(\bar{p})}|w|^{q}+2^{4-j}r^{-1}\th^*\mathcal{H}^{n}(\mathscr{B}_{4r}(\bar{p}))
\endaligned
\end{equation}
for $0\le q<1$ and $j\ge0$. Combining \eqref{Brjwqnaw} and \eqref{Brjwqnaw*}, we get
\begin{equation}\aligned\label{Brjwqnaw**}
\int_{\mathscr{B}_{r_{j+1}}(\bar{p})}|w|^{q}|\na w|\le\f{2^{j+4}}{r}\int_{\mathscr{B}_{r_{j}}(\bar{p})}|w|^{q}+2^{2q+4-j}(q+1)^{q}r^{-1}\th^*\mathcal{H}^{n}(\mathscr{B}_{4r}(\bar{p}))
\endaligned
\end{equation}
for $q\ge0$ and $j\ge0$.

Note that from Young's inequality \eqref{Young}, for $q\ge0$ one has
\begin{equation}\aligned\label{Young*}
2^2(q+1)|w|^{q}\le |w|^{q+1}+2^{2q+2}q^q.
\endaligned
\end{equation}
Note $r_{j}\le r$ for all $j\ge0$.
Combining Lemma \ref{SobMG} and \eqref{Brjwqnaw**}\eqref{Young*}, for $j\ge0$ and $q\ge0$, we have
\begin{equation}\aligned
&\left(\mathcal{H}^n(\mathscr{B}_{r_{j+2}}(\bar{p}))\right)^{\f1\a}\left(\int_{\mathscr{B}_{r_{j+2}}(\bar{p})}|w|^{\f {(q+1)\a}{\a-1}}\right)^{\f{\a-1}\a}\\
\le&\th r_{j+2}\left((q+1)\int_{\mathscr{B}_{r_{j+1}}(\bar{p})}|w|^{q}|\na w|+\f{2^{j+3}}{r}\int_{\mathscr{B}_{r_{j+2}}(\bar{p})}|w|^{q+1}\right)\\
\le&\th 2^{j+3}\left(2(q+1)\int_{\mathscr{B}_{r_{j}}(\bar{p})}|w|^{q}+2^{2q+1}(q+1)^{q+1}\th^*\mathcal{H}^{n}(\mathscr{B}_{4r}(\bar{p}))
+\int_{\mathscr{B}_{r_{j+2}}(\bar{p})}|w|^{q+1}\right)\\
\le&\th 2^{j+4}\left(\int_{\mathscr{B}_{r_{j}}(\bar{p})}|w|^{q+1}+2^{2q+2}(q+1)^{q+1}\th^*\mathcal{H}^{n}(\mathscr{B}_{4r}(\bar{p}))\right).
\endaligned
\end{equation}
In other words, there exists a constant $c_*$ depending only on $C_D,C_N$ such that
\begin{equation}\aligned
\Big|\Big||w|^{q}\Big|\Big|_{\f \a{\a-1},r_{j+2}}\le c_*2^j\Big|\Big||w|^{q}+2^{2q}q^q\Big|\Big|_{1,r_{j}}
\endaligned
\end{equation}
for any $j\ge0$, $q\ge1$, which implies
\begin{equation}\aligned\label{waqq-1r}
||w||_{\f{\a q}{\a-1} ,r_{j+2}}\le \left(c_*2^j\right)^{\f1q}\left(||w||_{q,r_{j}}+4q\right).
\endaligned
\end{equation}
Let $q_j=\left(\f\a{\a-1}\right)^j$ and $a_j=||w||_{q_j,r_{2j}}/q_j$ for $j\ge0$. Then from \eqref{waqq-1r} we have
\begin{equation}\aligned
||w||_{q_{j+1} ,r_{2j+2}}\le c_*^{\f1{q_j}}2^{\f {2j}{q_j}}\left(||w||_{q_j,r_{2j}}+4q_j\right),
\endaligned
\end{equation}
and
\begin{equation}\aligned\label{aj+1aj}
a_{j+1}\le \f{\a-1}{\a}c_*^{\f1{q_j}}2^{\f {2j}{q_j}}(a_j+4)
\endaligned
\end{equation}
for every $j\ge0$.
There is an integer $j_*>0$ depending only on $C_D,C_N$ such that
$c_*^{\f1{q_j}}2^{\f{2j}{q_j}}\le \f{\a+1}{\a}$ for all $j\ge j_*$. From \eqref{aj+1aj}, we get
\begin{equation}\aligned
a_{j+1}\le \f{\a^2-1}{\a^2}(a_j+4) \qquad \mathrm{for} \ \ j\ge j_*,
\endaligned
\end{equation}
which implies
\begin{equation}\aligned
a_{j+1}\le \max\{a_j,4(\a^2-1)\} \qquad \mathrm{for} \ \ j\ge j_*,
\endaligned
\end{equation}
Namely, $\max\{a_j,4(\a^2-1)\}$ is monotonic nonincreasing for $j\ge j_*$. Note that $a_0$ is bounded by a constant depending only on $C_D,C_N$ from \eqref{Br|w|}.
Hence there is a constant $c^*$ depending only on $C_D,C_N$ such that
\begin{equation}\aligned
a_j=\f{||w||_{q_j,r_{2j}}}{q_j}\le c^*
\endaligned
\end{equation}
for all $j\ge0$.
For each integer $k\ge 1$, there is an integer $j_k\ge0$ such that $q_{j_k}\le k\le q_{j_k+1}$. With H$\mathrm{\ddot{o}}$lder inequality, we have
\begin{equation}\aligned
||w||_{k,r_{\infty}}\le||w||_{k,r_{2j_k+2}}\le ||w||_{q_{j_k+1},r_{2j_k+2}}\le c^* q_{j_k+1}\le \f{c^*\a}{\a-1}k,
\endaligned
\end{equation}
where $r_\infty=\lim_{i\rightarrow\infty}r_i=\f r2$.
Therefore, by Stirling's formula
\begin{equation}\aligned
||w||_{k,r_{\infty}}^{k}\le \left(\f{c^*\a}{\a-1}\right)^{k} k^{k}\le \left(\f{c^*\a}{\a-1}\right)^{k}e^{k}k^{-\f12}k!,
\endaligned
\end{equation}
which implies
\begin{equation}\aligned
\f{1}{k!}\left|\left|\f{(\a-1)w}{c^*e\a}\right|\right|_{k,r_{\infty}}^{k}\le k^{-\f12}.
\endaligned
\end{equation}
Hence, there are constants $\la_*\in(0,\f{\a-1}{2c^*e\a})$ and $C_*>0$ depending only on $C_D,C_N$ such that
\begin{equation}\aligned
\fint_{\mathscr{B}_{\f r2}(\bar{p})} e^{\la_* w}\fint_{\mathscr{B}_{\f r2}(\bar{p})} e^{-\la_* w}\le C_*.
\endaligned
\end{equation}
Namely,
\begin{equation}\aligned
\fint_{\mathscr{B}_{\f r2}(\bar{p})} u^{\la_*}\fint_{\mathscr{B}_{\f r2}(\bar{p})} u^{-\la_*}\le C_*.
\endaligned
\end{equation}
Combining Theorem \ref{MVinequR}, we have
\begin{equation}\aligned\label{ula14r}
\sup_{\mathscr{B}_{\f14r}(\bar{p})} u^{\la_*}\le \left(C_{\la_*}2^{\f{2\a}{\la_*}}\right)^{\la_*}||u||_{\la_*,\f r2}^{\la_*}\le 2^{2\a}C_*C_{\la_*}^{\la_*}\left(\fint_{\mathscr{B}_{\f r2}(\bar{p})} u^{-\la_*}\right)^{-1}.
\endaligned
\end{equation}
Since
\begin{equation}\aligned
\De u^{-\la_*}=\la_*(\la_*+1)u^{-\la_*-2}|\na u|^2\ge0,
\endaligned
\end{equation}
then combining Theorem \ref{MVinequR} and \eqref{ula14r} we have
\begin{equation}\aligned
\sup_{\mathscr{B}_{\f r4}(\bar{p})}u^{-\la_*}\le C_12^{2\a}\fint_{\mathscr{B}_{\f r2}(\bar{p})} u^{-\la_*}\le 2^{4\a}C_1C_*C_{\la_*}^{\la_*}\left(\sup_{\mathscr{B}_{\f14r}(\bar{p})} u^{\la_*}\right)^{-1},
\endaligned
\end{equation}
which implies
\begin{equation}\aligned
\sup_{\mathscr{B}_{\f r4}(\bar{p})} u\le \left(2^{4\a}C_1C_*\right)^{1/\la_*}C_{\la_*}\inf_{\mathscr{B}_{\f r4}(\bar{p})}u.
\endaligned
\end{equation}
Hence for any $z\in \mathscr{B}_{2R}(\bar{p})$ we have
\begin{equation}\aligned
\sup_{\mathscr{B}_{\f R8}(z)} u\le \left(2^{4\a}C_1C_*\right)^{1/\la_*}C_{\la_*}\inf_{\mathscr{B}_{\f R8}(z)}u.
\endaligned
\end{equation}
This completes the proof.
\end{proof}

Using Theorem \ref{Harnack}, we can show the following Liouville type theorem for the minimal graphic function $u$ by considering $u-\inf_\Si u$.
\begin{theorem}
Let $\Si$ be an $n$-dimensional complete Riemannian manifold with \eqref{VD} and \eqref{NP}.
If $u$ is a positive minimal graphic function on $\Si$,
then $u$ is a constant.
\end{theorem}
As a corollary, if $\Si$ is an open manifold which is quasi isometric to an open manifold with nonnegative Ricci curvature,
then any positive minimal graphic function on $\Si$ is a constant.

\begin{remark} The Harnack's inequality \eqref{Harnack} implies H$\ddot{o}$lder continuity of the solution $u$ (see \cite{M} for instance).
Namely, there is a constant $\de>0$ depending only on $C_D,C_N$ such that for any entire minimal graphic function $u$ on $\Si$ satisfies
$$
\sup_{\mathscr{B}_{r}(\bar{p})}u-\inf_{\mathscr{B}_{R}(\bar{p})}u\le \f1\de\left(\f rR\right)^\de\left(\sup_{\mathscr{B}_{R}(\bar{p})}u-\inf_{\mathscr{B}_{R}(\bar{p})}u\right)
$$
for all $0<r<R<\infty$.
Therefore, the above Theorem can be improved somewhat to allow
$$\limsup_{R\rightarrow\infty}R^{-\de}\max\left\{-\inf_{B_R(p)}u,0\right\}=0.$$
\end{remark}

\section{Appendix}

Let $\Si$ be an $n$-dimensional complete manifold with Riemannian metric $\si$ and the Levi-Civita connection $D$.
Suppose that $\Si$ satisfies the volume doubling property \eqref{VD} and the uniform Neumann-Poincar$\mathrm{\acute{e}}$ inequality \eqref{NP}. Now let us show an isoperimetric inequality on $\Si$.

Let $\Om$ be an open set in $B_r(p)$ with rectifiable boundary.
For any $x\in\Om\setminus\p\Om$, there is a constant $r_x>0$ such that
\begin{equation}\aligned\label{Omrx12}
\mathcal{H}^n(\Om\cap B_{r_x}(x))=\f12\mathcal{H}^n(B_{r_x}(x)).
\endaligned
\end{equation}
From \eqref{VG}, we have
\begin{equation}\aligned
\mathcal{H}^n(\Om)\ge\f12\mathcal{H}^n(B_{r_x}(x))\ge\f1{2C_D}\mathcal{H}^n(B_{2r}(x))\left(\f{r_x}{2r}\right)^\a\ge\f1{2C_D}\mathcal{H}^n(B_{r}(p))\left(\f{r_x}{2r}\right)^\a,
\endaligned
\end{equation}
which implies
\begin{equation}\aligned\label{rxbound}
r_x\le 2r\left(\f{2C_D\mathcal{H}^n(\Om)}{\mathcal{H}^n(B_{r}(p))}\right)^{\f1\a}.
\endaligned
\end{equation}
By 5-lemma, there is a sequence of points $x_i\in\Om\setminus\p\Om$ such that $\Om\subset\bigcup_i B_{5r_{x_i}}(x_i)$ and $B_{r_{x_i}}(x_i)$ are mutually disjoint.
Then combining \eqref{VG}\eqref{NPOm}\eqref{Omrx12}\eqref{rxbound}, we get
\begin{equation}\aligned
&\mathcal{H}^n(\Om)\le\sum_i\mathcal{H}^n(B_{5r_{x_i}}(x_i))\le 5^\a C_D\sum_i\mathcal{H}^n(B_{r_{x_i}}(x_i))\\
=&2\times5^\a C_D\sum_i\mathcal{H}^n(\Om\cap B_{r_{x_i}}(x_i))
\le2\times 5^\a C_DC_N\sum_ir_{x_i}\mathcal{H}^{n-1}(\p\Om\cap B_{r_{x_i}}(x_i))\\
\le&4\times5^\a C_DC_Nr\left(\f{2C_D\mathcal{H}^n(\Om)}{\mathcal{H}^n(B_{r}(p))}\right)^{\f1\a}\sum_i\mathcal{H}^{n-1}(\p\Om\cap B_{r_{x_i}}(x_i))\\
\le&4\times5^\a C_DC_Nr\left(\f{2C_D\mathcal{H}^n(\Om)}{\mathcal{H}^n(B_{r}(p))}\right)^{\f1\a}\mathcal{H}^{n-1}(\p\Om).
\endaligned
\end{equation}
Note $\a\ge2$ from $\eqref{alog2CD}$. Then the above inequality implies an isoperimetric inequality:
\begin{equation}\aligned
\left(\mathcal{H}^n(\Om)\right)^{1-\f1\a}\le 8\times5^\a C_D^2C_N\left(\mathcal{H}^n(B_{r}(p))\right)^{-\f1\a}r\mathcal{H}^{n-1}(\p\Om).
\endaligned
\end{equation}

\bibliographystyle{amsplain}

\end{document}